\documentclass{article}
\usepackage[final,nonatbib]{nips_2016_v2}
\usepackage{amsthm,amsmath,amssymb,dsfont}
\usepackage{graphicx,xcolor}
\usepackage{bbm}
\usepackage{subcaption}
\usepackage{nicefrac}
\usepackage{setspace}
\usepackage[compact]{titlesec}
\usepackage[font=footnotesize,labelfont=bf]{caption}

\newtheorem{theorem}{Theorem}[section]

\newtheorem{assumption}{Assumption}[section]
\newtheorem{lemma}[theorem]{Lemma}

\DeclareMathOperator*{\argmin}{\arg\!\min}
\DeclareMathOperator*{\argmax}{\arg\!\max}

\usepackage{geometry}
\def\grad{\nabla}

\def\bb{\mathbf{b}}

\def\be{\mathbf{e}}

\def\bh{\mathbf{h}}

\def\bp{\mathbf{p}}
\def\bq{\mathbf{q}}

\def\bs{\mathbf{s}}

\def\bv{\mathbf{v}}
\def\bw{\mathbf{w}}
\def\bx{\mathbf{x}}  
\def\by{\mathbf{y}}
\def\bz{\mathbf{z}}

\def\bD{\mathbf{D}}

\def\bI{\mathbf{I}}

\def\bQ{\mathbf{Q}}

\def\cB{\mathcal{B}}
\def\cC{\mathcal{C}}
\def\cD{\mathcal{D}}
\def\cE{\mathcal{E}}

\def\cG{\mathcal{G}}

\def\cI{\mathcal{I}}

\def\cK{\mathcal{K}}
\def\cL{\mathcal{L}}

\def\cN{\mathcal{N}}
\def\cO{\mathcal{O}}
\def\cP{\mathcal{P}}

\def\cR{\mathcal{R}}
\def\cS{\mathcal{S}}

\def\cX{\mathcal{X}}
\def\cY{\mathcal{Y}}

\def\smskip{\smallskip}

\def\texitem#1{\par\smskip\noindent\hangindent 25pt
               \hbox to 25pt {\hss #1 ~}\ignorespaces}

\def\norm#1{\left\|#1\right\|}

\newcommand{\BEAS}{\begin{eqnarray*}}
\newcommand{\EEAS}{\end{eqnarray*}}
\newcommand{\BEA}{\begin{eqnarray}}
\newcommand{\EEA}{\end{eqnarray}}
\newcommand{\BEQ}{\begin{eqnarray}}
\newcommand{\EEQ}{\end{eqnarray}}
\newcommand{\BIT}{\begin{itemize}}
\newcommand{\EIT}{\end{itemize}}
\newcommand{\BNUM}{\begin{enumerate}}
\newcommand{\ENUM}{\end{enumerate}}

\newcommand{\BA}{\begin{array}}
\newcommand{\EA}{\end{array}}

\newcommand{\reals}{\mathbb{R}}
\newcommand{\integers}{\mathbb{Z}}

\newcommand{\diag}{\mathop{\bf diag}}

\newcommand{\dom}{\mathop{\bf dom}}

\newcommand{\intr}{\mathop{\bf int}}
\newcommand{\relint}{\mathop{\bf rel int}}

\newif\ifpagenumbering
\pagenumberingtrue

\pagenumberingfalse

\newsavebox{\theorembox}
\newsavebox{\lemmabox}
\newsavebox{\defnbox}
\newsavebox{\corollarybox}
\newsavebox{\remarkbox}
\newsavebox{\assbox}
\savebox{\theorembox}{\noindent\bf Theorem}
\savebox{\lemmabox}{\noindent\bf Lemma}
\savebox{\defnbox}{\noindent\bf Definition}
\savebox{\corollarybox}{\noindent\bf Corollary}
\savebox{\remarkbox}{\noindent\bf Remark}
\savebox{\assbox}{\noindent\bf Assumption}
\newtheorem{remark}{\usebox{\remarkbox}}[section]
\newtheorem{defn}{\usebox{\defnbox}}

\def\fprod#1{\left\langle#1\right\rangle}
\def\prox#1{\mathbf{prox}_{#1}}
\def\ind#1{\mathds{1}_{#1}}
\def\T{\mathsf{T}}
\def\id{\mathbf{I}}
\def\zero{\mathbf{0}}
\def\one{\mathbf{1}}

\def\btheta{\boldsymbol{\theta}}
\def\blambda{\boldsymbol{\lambda}}
\def\bmu{\boldsymbol{\mu}}
\def\bxi{\boldsymbol{\xi}}

\title{A primal-dual method for conic constrained distributed optimization problems}
\begin{document}
\author{
Necdet Serhat Aybat\\
Department of Industrial Engineering\\
  Penn State University\\
  University Park, PA 16802 \\
  \texttt{nsa10@psu.edu} \\
\And
Erfan Yazdandoost Hamedani\\
Department of Industrial Engineering\\
  Penn State University\\
  University Park, PA 16802 \\
  \texttt{evy5047@psu.edu}
}
\maketitle
\begin{abstract}
We consider cooperative multi-agent consensus optimization problems over an undirected network of agents, where only those agents connected by an edge can directly communicate. The objective is to minimize the sum of agent-specific composite convex functions over agent-specific private conic constraint sets; hence, the optimal consensus decision should lie in the intersection of these private sets. We provide convergence rates both in sub-optimality, infeasibility and consensus violation; examine the effect of underlying network topology on the convergence rates of the proposed decentralized algorithms; and show how to extend these methods to handle time-varying communications networks and to solve problems with resource sharing constraints.
\end{abstract}
\section{Introduction}
Let $\cG=(\cN,\cE)$ denote a \emph{connected} undirected graph of $N$ computing nodes, where $\cN\triangleq\{1,\ldots,N\}$ and $\mathcal{E}\subseteq \mathcal{N}\times \mathcal{N}$ denotes the set of edges -- without loss of generality assume that $(i,j)\in \mathcal{E}$ implies $i< j$. Suppose nodes $i$ and $j$ can exchange information only if $(i,j) \in \cE$, and each node $i\in\cN$ has a \emph{private} (local) cost function $\Phi_i:\reals^n\rightarrow\reals\cup\{+\infty\}$ such that
{\small
\begin{equation}
\label{eq:F_i}
\Phi_i(x)\triangleq \rho_i(x) + f_i(x),
\end{equation}
}%
where $\rho_i: \mathbb{R}^n \rightarrow \mathbb{R}\cup\{+\infty\}$ is a possibly \emph{non-smooth} convex function, and $f_i: \mathbb{R}^n \rightarrow \mathbb{R}$ is a \emph{smooth} convex function. We assume that $f_i$ is differentiable on an open set containing $\dom \rho_i$ with a Lipschitz continuous gradient $\grad f_i$, of which Lipschitz constant is $L_i$; and the prox map of $\rho_i$, 
{\small
\begin{equation}
\label{eq:prox}
\prox{\rho_i}(x)\triangleq\argmin_{y \in \reals^n} \left\{ \rho_i(y)+\tfrac{1}{2}\norm{y-x}^2 \right\},
\end{equation}
}%
is \emph{efficiently} computable for $i\in\cN$, where $\norm{.}$ denotes the Euclidean norm. Let $\mathcal{O}_i\triangleq\{j\in\mathcal{N} : (i, j) \in \mathcal{E} \text{ or } (j, i) \in \mathcal{E}\}$ denote the set of neighboring nodes of $i \in \mathcal{N}$, and $d_i\triangleq |\mathcal{O}_i|$ is the degree of node $i\in \mathcal{N}$; we also define $\mathcal{N}_i\triangleq \mathcal{O}_i\cup \{i\}$. Consider the following minimization problem:
{\small
\begin{align}\label{eq:central_problem}
\min_{x\in\reals^n}\ \sum_{i\in \mathcal{N}}\Phi_i(x)\quad \hbox{s.t.}\quad A_ix -b_i \in \mathcal{K}_i,\quad \forall{i}\in\mathcal{N},
\end{align}
}%
\noindent where $A_i\in \mathbb{R}^{m_i\times n}$, $b_i\in \mathbb{R}^{m_i}$ and $\mathcal{K}_i\subseteq{R}^{m_i}$ is a closed, convex cone. Suppose that projections onto $\cK_i$ can be computed efficiently, while the projection onto the preimage $A_i^{-1}(\cK_i+b_i)$ is assumed to be \emph{impractical}, e.g., when $\cK_i$ is the positive semidefinite cone, projection to preimage requires solving an SDP. Our objective is to solve \eqref{eq:central_problem} in a decentralized fashion using the computing nodes $\cN$ and exchanging information only along the edges $\cE$. In Section~\ref{sec:static} and Section~\ref{sec:dynamic}, we consider \eqref{eq:central_problem} when the topology of the connectivity graph is \emph{static} and \emph{time-varying}, respectively. In Section~\ref{sec:dual}, we show that \emph{resource allocation} type problems of the following form
{\small
\begin{align}\label{eq:dual-implement-intro}
\min_{\{\xi_i\}_{i\in\cN}}\ \sum_{i\in \cN}{\Phi_i(x,\xi_i)}\quad\hbox{s.t.}\quad \sum_{i\in\cN}Q_i x+R_i\xi_i-r_i \in \cK,
\end{align}}%
can be handled in a similar way using both primal and dual consensus iterations at the same time, where $\cK\subseteq \reals^m$ is a closed convex cone, $R_i\in\reals^{m\times n_i}$, $Q_i\in\reals^{m\times n}$, and $r_i\in \reals^{m}$ are the problem data such that each node $i\in\cN$ only have access to $R_i$, $r_i$ and $\cK$ along with its objective $\Phi_i(x,\xi_i)\triangleq\rho(x,\xi_i)+f(x,\xi_i)$ defined similarly as in \eqref{eq:F_i}.

This computational setting, i.e., decentralized consensus optimization, appears as a generic model for various applications in signal processing, e.g., \cite{ling2010decentralized,schizas2008consensus}, machine learning, e.g.,~\cite{forero2010consensus,mcdonald2010distributed,yan2013distributed} and statistical inference, e.g.,~\cite{mateos2010distributed}. Clearly, \eqref{eq:central_problem} and \eqref{eq:dual-implement-intro} can also be solved in a ``centralized" fashion by communicating all the private functions $\Phi_i$ to a \emph{central} node, and solving the overall problem at this
node. However, such an approach can be very expensive both from communication
and computation perspectives when compared to the distributed algorithms which are far more scalable to increasing problem data and network sizes. In particular, suppose $(A_i,b_i) \in \reals^{m \times (n+1)}$ and $\Phi_i(x) = \norm{A_i x - b_i}_2^2 +\lambda \norm{x}_1$ for some given $\lambda>0$ for $i\in\cN$ such that $m\ll n$ and $N\gg 1$. Hence, \eqref{eq:central_problem} is a very large scale LASSO problem with \emph{distributed} data. 
To solve~\eqref{eq:central_problem} in a centralized fashion, the data
$\{(A_i,b_i): i \in \cN\}$ needs to be communicated to the central node. This can be
prohibitively expensive, and may also violate privacy constraints -- in case some node $i$
 does not want to reveal the details of its \emph{private} data.
Furthermore, it requires that the central node has large enough memory to
be able to accommodate all the data. On the other hand, at the expense of slower
convergence, one can completely do away with a central node, and seek for
\emph{consensus} among all the nodes on an optimal decision
using ``local" decisions communicated by the neighboring nodes. From computational
 perspective, for certain cases, computing partial gradients \emph{locally} 
can be more computationally efficient
when compared to computing the entire gradient at a central node. 
These
considerations in mind, we
propose decentralized algorithms that can compute
solutions to~\eqref{eq:central_problem} and \eqref{eq:dual-implement-intro} using only local computations without explicitly requiring
the nodes to communicate the functions $\{\Phi_i: i \in \cN\}$; thereby,
circumventing all privacy, communication and memory issues. Examples of \emph{constrained} machine learning problems that fit into our framework include multiple kernel learning~\cite{bach2004multiple}, and primal linear support vector machine~(SVM) problems. In the numerical section we implemented the proposed algorithms on the primal SVM problem.
\subsection{Previous Work}
There has been active research~\cite{nedic2009subgradient,chambolle2011first,he2012convergence,chambolle2015ergodic,chen2014optimal} on solving convex-concave saddle point problems $\min_x\max_y\cL(x,y)$. In \cite{nedic2009subgradient} a subgradient method is proposed for computing a saddle point for $\cL(x,y)$, where $(x^*,y^*)$ is called a saddle point if $\cL(x^*,y)\leq\cL(x^*,y^*)\leq \cL(x,y^*)$ for all $x,y$. Assuming that subgradients and primal-dual iterates generated by their method are uniformly bounded, the authors showed that $|\cL(\bar{x}^k,\bar{y}^k)-\cL(x^*,y^*)|\leq\cO(\frac{1}{\alpha k})+\alpha L^2$,
where $\{(\bar{x}^k,\bar{y}^k)\}$ denotes the  primal-dual ergodic average sequence, $\alpha$ is the constant step-size and $L$ is a uniform bound on the norm of generated subgradients. Note that due to constant error level $\alpha L^2$, setting $\alpha=\cO(\epsilon)$, one can compute an approximate saddle point with $\epsilon$ error in function values within $\cO(1/\epsilon^2)$ iterations; that being said, further iterations won't improve the quality due to constant step size. They also showed that under Slater's condition, the boundedness assumption on iterate sequence can be relaxed, and the method can inexactly solve constrained convex optimization problems with guarantees on the amount of feasibility violation and on the primal objective function values at approximate solutions. 
In~\cite{chambolle2011first} primal-dual proximal algorithms are proposed for convex-concave problems with known saddle-point structure $\min_x\max_y \cL_s(x,y)\triangleq\Phi(x)+\fprod{Tx,y}-h(y)$, where $\Phi$ and $h$ are convex functions, and $T$ is a linear map. These algorithms converge with rate $\mathcal{O}(1/k)$ for the primal-dual gap, 
and they can be modified to yield a convergence rate of $\mathcal{O}(1/k^2)$ when either $\Phi$ or $h$ is strongly convex, and $\mathcal{O}(1/e^k )$ linear rate, when both $\Phi$ and $h$ are strongly convex. In~\cite{he2012convergence} a primal-dual contraction method is presented which simplifies the existing convergence analysis of primal-dual methods for structured saddle point problems as in~\cite{chambolle2011first}. Moreover, in a similar setting to~\cite{chambolle2011first}, when $\Phi$ is smooth and $\grad \Phi$ is Lipschitz continuous with constant $L$, the primal-dual method proposed in~\cite{chen2014optimal} yields an ``optimal" rate of $\cO(L/k^2+\norm{T}/k)$. More recently, in~\cite{chambolle2015ergodic} Chambolle and Pock extend their previous work in~\cite{chambolle2011first}, using simpler proofs, to handle composite convex primal functions $\Phi(x)=\rho(x)+f(x)$, and to deal with proximity operators based on Bregman distance functions, where $\rho$ and $f$ are both convex and $f$ is smooth. In various other papers, e.g., \cite{chen2013primal,condat2013primal}, primal-dual approaches for solving convex minimization problems containing both smooth and nonsmooth functions were studied as well. 

Consider $\min_{x\in\reals^n}\{\sum_{i\in\cN}\Phi_i(x):\ x\in\cap_{i\in\cN}\cX_i\}$ over $\cG=(\cN,\cE)$.  Although the \emph{unconstrained} consensus optimization, i.e., $\cX_i=\reals^n$, is well studied -- see~\cite{Nedic08_1J,nedic2014distributed} and the references therein,
the \emph{constrained} case 
is still an immature, and recently developing area of active research
\cite{Nedic08_1J,lee2013optimization,nedic2014distributed,srivastava2012distributed,chang2014distributed,lee2013distributed,mateos2015distributed,yuan2015regularized,yuan2011distributed,zhu2012distributed,nedic2010constrained,srivastava2010distributed,liu2016collective}.
 Other than few exceptions~\cite{chang2014distributed,lee2013distributed,mateos2015distributed,yuan2015regularized,yuan2011distributed,zhu2012distributed}, the methods in literature require that each node compute a projection on the privately known set $\cX_i$ in addition to consensus and (sub)gradient steps, e.g., \cite{srivastava2012distributed,nedic2010constrained,srivastava2010distributed}. Moreover, among those few exceptions that do not use projections onto $\cX_i$ when $\Pi_{\cX_i}$ is not easy to compute, only \cite{chang2014distributed,mateos2015distributed} can handle agent-specific constraints without assuming global knowledge of the constraints by all agents. However, \emph{no} rate results in terms of suboptimality, local infeasibility, and consensus violation exist for the primal-dual distributed methods in~\cite{chang2014distributed,mateos2015distributed} when implemented for the agent-specific conic constraint sets $\cX_i=\{x:A_ix-b_i\in \cK_i\}$ studied in this paper. To sum up, there is a pressing need for new computational approaches to reach a consensus among the distributed computing nodes on an optimal decision when there are \emph{private} node-specific constraints.

In~\cite{nedic2010constrained} authors considered minimizing sum of privately known convex functions, $f_i$, over the intersection of node specific, privately known convex sets $\cX_i$, i.e., $(P):\ \min_x\{\sum_{i\in\cN}f_i(x):\ x\in\bigcap_{i\in\cN}\cX_i\}$. They proposed a distributed projected subgradient algorithm, and analyzed its convergence considering two different scenarios: i) the connectivity graph is time-varying, all local constraint sets are the same, i.e., $\cX_i=\cX$ for $i\in\cN$, and subdifferential of each $f_i$ is uniformly bounded on $\cX$; ii) the communication graph is static and it is fully connected, $\cX_i$'s are compact, and may be different at each node. They showed that the generated iterates converge to an optimal solution without providing any rate result. In~\cite{srivastava2010distributed}, the authors considered the constrained consensus optimization problem (P) as in~\cite{nedic2010constrained}, over time-varying communication networks. Assuming $\cX_i$ is compact for $i\in\cN$, it is shown that the proposed projected subgradient method using square summable but not summable step-size sequence converges almost surely even when there is noise in communication links, and subgradient evaluations are corrupted with bounded stochastic error.
Note that both algorithms~\cite{nedic2010constrained,srivastava2010distributed} require computing projections onto the local sets, and they may not be implemented efficiently if the local sets do not assume simple projections, which is the case considered in our paper. In order to alleviate this issue to some extend, the distributed random projection (DRP) algorithm was presented in \cite{lee2013distributed} for solving (P) when $\cX_i=\bigcap_{j\in\cI_i}\cX_i^j$, where for each $i\in\cN$, $\cX_i^j$ is convex for all $j$ in a finite index set $\cI_i$. Suppose that computing the projection onto any one of the components, $\cX_i^j$, is easy, even though the projection onto whole local constraint set $\cX_i$ might still be computationally expensive. At $k$-th iteration, for each node $i\in\cN$, after the subgradient is computed, the algorithm computes a projection on a random component $\cX_i^{\Omega_i(k)}$ of the local set $\cX_i$, where $\Omega_i(k)\in\cI_i$ is random set process satisfying certain assumptions. Assuming that each local objective function $f_i$ is Lipschitz-differentiable with bounded gradients over $\cX_i$, they proved that for square summable but not summable step-sizes, the generated iterates converge to an optimal solution almost surely. In addition, asynchronous gossip-based random projection algorithm was proposed in~\cite{lee2013optimization} which uses gossip type communication and local computation.

Next, we will briefly review some important work on constrained distributed optimization which does not heavily rely on computing local projections. Authors in~\cite{zhu2012distributed} considered consensus optimization problems where the consensus decision should lie in the intersection of some privately known convex sets, i.e., $\cX=\bigcap_{i\in\cN}\cX_i$, such that $\cX$ is compact, and should satisfy the \emph{globally} known convex inequality and equality constraints -- all the local functions defining the objective and constraints are merely convex. Assuming a Slater point exists, they proposed two distributed primal-dual subgradient algorithms using square summable but non-summable step-size sequence: one algorithm is designed for the case where the equality constraint is absent, and the other one for the case where the local constraint sets are identical. For time-varying communication network topology, convergence of the proposed algorithms is shown, but without providing any rate result.
In \cite{yuan2011distributed}, the distributed primal-dual subgradient method~(DPDSM) is proposed for minimizing the sum of local convex functions subject to \emph{globally known} constraints without assuming differentiability. DPDSM can compute approximate saddle points of the Lagrangian function when the connectivity network topology is static. The authors show that it can find the optimal value within error level $\cO(\alpha)$ depending on the constant step-size $\alpha$ chosen. More specifically, DPDSM is a multi-step consensus method, i.e., in each iteration neighboring nodes exchange information multiple times; given an error level $\epsilon$, choosing step-size $\alpha=\cO(\epsilon)$ and assuming that the number of consensus steps per iteration is sufficiently large, the convergence rate in terms of suboptimality and infeasibility is $\mathcal{O}(\frac{1}{\alpha k}+ C\alpha)$ in the ergodic sense for some constant $C$, and consensus violation is also $\cO(\alpha)$. 
The number of consensus steps per iteration depends on connectivity of the graph, i.e., if the network graph has weak connectivity, then the algorithm converges very slowly. 
In a follow-up paper~\cite{yuan2015regularized}, the authors consider minimizing sum of local convex functions subject to a \emph{globally} known inequality constraint defined by a (possibly non-smooth) convex function $c:\reals^n\rightarrow\reals$, i.e., $\min_x\sum_{i\in\cN}f_i(x)\ \hbox{s.t.}\ x\in\cX=\{x:\ c(x)\leq 0\}$, but this time over a time-varying communication network. They assumed that i) $\cX$ is contained in a ball $\cB$; ii) the subdifferentials $\partial f_i(x)$ and $\partial c(x)$ are uniformly bounded for all $x\in\cB$, iii) $|c(x)|$ is bounded over $\cB$,
iv) $\min_x\{\|g_c(x)\|_2:\ c(x)=0,\ g_c(x)\in\partial c(x)\}\geq \rho$ for some constant $\rho>0$. This method does not require computing projections onto $\cX$ in any iteration, but the last one -- it avoids these projections by appropriately regularizing the Lagrangian function depending on the infeasibility of the iterates, and it computes only one projection onto $\cX$ at the very last iteration to achieve feasibility with respect to $\cX$, while consensus may still be violated. For this distributed regularized primal-dual subgradient method, it is shown that the suboptimality and consensus violation decrease with $O(1/k^{1 \over 4})$ and with $\cO(1/k^{1 \over 2})$ rates, respectively, where $k$ is the consensus iteration counter. 

In~\cite{chang2014distributed}, a general setting for constrained distributed optimization in a time-varying network topology has been considered where the objective is to minimize a composition of a global network function (smooth) with the summation of local objective functions (smooth), subject to inequality constraints on the summation of agent specific constrained functions and local compact sets. They proposed a consensus-based distributed primal-dual perturbation (PDP) algorithm using a square summable but not summable step-size sequence, and showed that the local primal-dual iterates converges to a global optimal primal-dual solution; however, no rate result was provided. The proposed PDP method can also handle non-smooth constraints with similar convergence guarantees. Finally, while we were preparing this paper, we became aware of a very recent work~\cite{mateos2015distributed} related to ours in certain ways: i) Fenchel conjugation and Laplacian averaging play key roles in both papers, ii) dual consensus is used to decompose separable constraints (see Section~\ref{sec:dual} in our paper). The authors proposed a distributed algorithm on time-varying communication network for solving saddle-point problems subject to consensus constraints. The algorithm can also be applied to solve consensus optimization problems with inequality constraints that can be written as summation of local convex functions of local and global variables. 
Assuming i) each agent's local variable lies in locally known compact convex set, and the global variable lies in a globally known compact convex set, ii) a ball centered at the origin is given such that it contains the optimal dual solution set, iii) subgradients generated by the algorithm are bounded, it is shown that using a carefully selected decreasing step-size sequence, the ergodic average of primal-dual sequence converges with $\mathcal{O}(1/\sqrt{k})$ rate in terms of saddle-point evaluation error; however, when applied to constrained optimization problems, no rate in terms of suboptimality and infeasibility is provided. 


Before concluding this section, we would like to mention a related work in unconstrained consensus optimization which inspired the analysis in our paper. Authors in~\cite{chen2012fast} worked on an \emph{unconstrained} consensus optimization problem over network with time-varying  topology. The objective function is the summation of local convex composite functions such that each local function is of the form $\Phi_i=\rho+f_i$ as in \eqref{eq:F_i}. They assumed that non-smooth function $\rho$ is common to all nodes and smooth functions, $f_i$, have bounded gradients. They proved that their proposed inexact proximal-gradient method, which exploits the function structure, can compute $\epsilon$-feasible and $\epsilon$-optimal solution in $\mathcal{O}(1/\sqrt{\epsilon})$ iterations which require $k$ communication rounds with neighbors during the $k$-th iteration, hence; hence, leading to $\mathcal{O}(1/\epsilon)$ total number of communication rounds in total. That said, there are many practical problems where nodes in the network have different non-smooth components in their objective and/or have node specific constraints. These important concerns motivated our paper, and it is shown that the number of communication rounds per iteration can be reduced to $\sqrt[p]{k}$ per node for some $p\geq 1$, leading to $\cO(1/\epsilon^{1+{1 \over p}})$ rate for a more general class of problems given in \eqref{eq:central_problem} and \eqref{eq:dual-implement-intro}. 
%
%

\textbf{Contribution.} We propose primal-dual algorithms for distributed optimization subject to agent specific conic constraints and/or global conic constraints with separable local components. By assuming composite convex structure on the primal functions, we show that our proposed algorithms converge with $\mathcal{O}(1/{k})$ rate where $k$ is the number of consensus iterations. To the best of our knowledge, this is the best rate result for our setting. Indeed, $\epsilon$-optimal and $\epsilon$-feasible solution can be computed within $\cO(1/\epsilon)$ consensus iterations for the static topology, and within $\cO(1/\epsilon^{1+\nicefrac{1}{p}})$ consensus iterations for the dynamic topology for any rational $p\geq 1$, although $\cO(1)$ constant gets larger for large $p$. Moreover, these methods are fully distributed, i.e., the agents are \emph{not} required to know any global parameter depending on the entire network topology, e.g., the second smallest eigenvalue of the Laplacian; instead, we only assume that agents know who their neighbors are. 
\subsection{Preliminary}
Let $\cX$ and $\cY$ be finite-dimensional vector spaces. In a recent paper, Chambolle and Pock~\cite{chambolle2015ergodic} proposed a primal-dual algorithm~(PDA) for the following convex-concave saddle-point problem:
{\small
\begin{align}
\label{eq:saddle-problem}
\min_{\bx\in\cX}\max_{\by\in\cY}\cL(\bx,\by)\triangleq\Phi(\bx)+\fprod{T\bx,\by}-h(\by),\quad \hbox{\normalsize where}\quad \Phi(\bx)\triangleq\rho(\bx)+f(\bx),
\end{align}
}%
where $\rho$ and $h$ are possibly non-smooth convex functions, $f$ is a convex function and has a Lipschitz continuous gradient defined on $\dom \rho$ with constant $L$. Briefly, given $\bx^0,\by^0$ and algorithm parameters $\nu_x,\nu_y>0$, PDA consists of two proximal-gradient steps: 
{\small
\begin{align}
\bx^{k+1}&\gets\argmin_{\bx} \rho(\bx)+f(\bx^k)+\fprod{\grad f(\bx^k),~\bx-\bx^k}+\fprod{T\bx,\by^k}+\frac{1}{\nu_x}D_x(\bx,\bx^k)
\label{eq:PDA-x}\\
\by^{k+1}&\gets\argmin_{\by} h(\by)-\fprod{T(2\bx^{k+1}-\bx^k),\by}+\frac{1}{\nu_y}D_y(\by,\by^k),
\label{eq:PDA-y}
\end{align}
}%
where $D_x$ and $D_y$ are Bregman distance functions corresponding to some continuously differentiable strongly convex functions $\psi_x$ and $\psi_y$ such that $\dom \psi_x\supset\dom\rho$ and $\dom \psi_y\supset\dom h$. In particular, $D_x(\bx,\bar{\bx})\triangleq\psi_x(\bx)-\psi_x(\bar{\bx})-\fprod{\grad \psi_x(\bar{\bx}),~\bx-\bar{\bx}}$, and $D_y$ is defined similarly. In~\cite{chambolle2015ergodic}, a simple proof for the ergodic convergence is provided; indeed, it is shown that, when the convexity modulus for $\psi_x$ and $\psi_y$ is 1, if $\tau,\kappa>0$ are chosen such that $(\frac{1}{\nu_x}-L)\frac{1}{\nu_y}\geq\sigma^2_{\max}(T)$, then
{\small
\begin{equation}
\label{eq:DPA-rate}
\cL(\bar{\bx}^K,\by)-\cL(\bx,\bar{\by}^K)\leq\frac{1}{K}\left(\frac{1}{\nu_x}D_x(\bx,\bx^0)+\frac{1}{\nu_y}D_y(\by,\by^0)-\fprod{T(\bx-\bx^0),\by-\by^0}\right),
\end{equation}
}%
for all $\bx,\by\in\cX\times\cY$, where $\bar{\bx}^K\triangleq\frac{1}{K}\sum_{k=1}^K \bx^k$ and $\bar{\by}^K\triangleq\frac{1}{K}\sum_{k=1}^K \by^k$.

First, we define the notations 
used throughout the paper. Next, in Theorem~\ref{thm:general-rate}, we discuss a special case of \eqref{eq:saddle-problem}, which will help us prove the main results of this paper, and also allow us to develop decentralized algorithms for the consensus optimization problem in~\eqref{eq:central_problem}. The proposed algorithms in this paper can distribute the computation over the nodes such that each node's computation is based on the local topology of $\cG$ and the private information only available to that node.

\textbf{Notations.} Throughout the paper, $\norm{.}$ denotes the Euclidean norm. Given a convex set $\cS$, let $\sigma_{\cS}(.)$ denote its support function, i.e., $\sigma_{\cS}(\theta)\triangleq\sup_{w\in \cS}\fprod{\theta,~w}$, let $\ind{S}(\cdot)$ denote the indicator function of $\cS$, i.e., $\ind{S}(w)=0$ for $w\in\cS$ and equal to $+\infty$ otherwise, and let $\cP_{\cS}(w)\triangleq\argmin\{\norm{v-w}:\ v\in\cS\}$ denote the projection onto $\cS$. For a closed convex set $\cS$, we define the distance function as $d_{\cS}(w)\triangleq\norm{\cP_{\cS}(w)-w}$. Given a convex cone $\cK\in\reals^m$, let $\mathcal{K}^*$ denote its dual cone, i.e., $\mathcal{K}^*\triangleq\{\theta\in\reals^{m}: \ \langle \theta,w\rangle \geq 0\ \ \forall w\in\mathcal{K}\}$, and $\cK^\circ\triangleq -\cK^*$ denote the polar cone of $\cK$. Note that for a given cone $\cK\in\reals^m$, $\sigma_{\cK}(\theta)=0$ for $\theta\in\cK^\circ$ and equal to $+\infty$ if $\theta\not\in\cK^\circ$, i.e., $\sigma_{\cK}(\theta)=\ind{\cK^\circ}(\theta)$ for all $\theta\in\reals^m$.
Cone $\cK$ is called \emph{proper} if it is closed, convex, pointed, and it has a nonempty interior. Given a convex function $g:\reals^n\rightarrow\reals\cup\{+\infty\}$, its convex conjugate is defined as $g^*(w)\triangleq\sup_{\theta\in\reals^n}\fprod{w,\theta}-g(\theta)$. $\otimes$ denotes the Kronecker product, and $\id_n$ is the $n\times n$ identity matrix.

\begin{defn}
\label{def:bregman}
Let $\cX\triangleq\Pi_{i\in\cN}\reals^n$ and $\cX\ni\bx=[x_i]_{i\in\cN}$; and $\cY\triangleq\Pi_{i\in\cN}\reals^{m_i}\times\reals^{m_0}$,  $\cY\ni\by=[\btheta^\top \blambda^\top]^\top$ and $\btheta=[\theta_i]_{i\in\cN}\in\reals^m$, where $m\triangleq\sum_{i\in\cN}{m_i}$, and $\Pi$ denotes the Cartesian product. Given parameters $\gamma>0$, $\kappa_i,\tau_i>0$ for $i\in\cN$, let $\mathbf{D}_\gamma\triangleq\frac{1}{\gamma}\id_{m_0}$, $\mathbf{D}_\kappa\triangleq\diag([\frac{1}{\kappa_i}\id_{m_i}]_{i\in\cN})$, and $\mathbf{D}_\tau\triangleq\diag([\frac{1}{\tau_i}\id_n]_{i\in\cN})$. Defining $\psi_x(\bx)\triangleq\frac{1}{2}\bx^\top\mathbf{D}_\tau\bx$ and $\psi_y(\by)\triangleq\frac{1}{2}\btheta^\top\mathbf{D}_\kappa\btheta+\frac{1}{2}\blambda^\top\mathbf{D}_\gamma\blambda$ leads to the following Bregman distance functions: $D_x(\bx,\bar{\bx})=\frac{1}{2}\norm{\bx-\bar{\bx}}_{\mathbf{D}_\tau}^2$, and $D_y(\by,\bar{\by})=\frac{1}{2}\norm{\btheta-\bar{\btheta}}_{\mathbf{D}_\kappa}^2+\frac{1}{2}\norm{\blambda-\bar{\blambda}}_{\mathbf{D}_\gamma}^2$, where the $Q$-norm is defined as $\norm{z}_Q\triangleq (z^\top Q z)^{\tfrac{1}{2}}$ for 
$Q\succ 0$.
\end{defn}
\begin{theorem}
\label{thm:general-rate}
Let $\cX$, $\cY$, and Bregman functions $D_x$, $D_y$ are defined as in Definition~\ref{def:bregman}. Suppose $\Phi(\bx)\triangleq\sum_{i\in\cN}\Phi_i(x_i)$, and $h(\by)\triangleq h_0(\blambda)+\sum_{i\in\cN}h_i(\theta_i)$, where $\{\Phi_i\}_{i\in\cN}$ are composite convex functions defined as in~\eqref{eq:F_i}, and $\{h_i\}_{i\in\cN}$ are closed convex functions with simple prox-maps. Given $A_0\in\reals^{m_0\times n|\cN|}$ and $\{A_i\}_{i\in\cN}$ such that $A_i\in\reals^{m_i\times n}$, let $T=[A^\top~ A_0^\top]^\top$, where $A\triangleq\diag([A_i]_{i\in\mathcal{N}})\in\reals^{m \times n|\cN|}$ is a block-diagonal matrix. Given the initial point $(\bx^0,\by^0)$, the PDA iterate sequence $\{\bx^k,\by^k\}_{k\geq 1}$, generated according to \eqref{eq:PDA-x} and \eqref{eq:PDA-y} when $\nu_x=\nu_y=1$, satisfies \eqref{eq:DPA-rate} for all $K\geq 1$ if
{\small $\mathbf{\bar{Q}}(A,A_0)\triangleq\begin{bmatrix}
\mathbf{\bar{D}}_\tau &-A^\top& -A_0^\top\\
-A & \mathbf{D}_\kappa & 0\\
-A_0 & 0 & \mathbf{D}_\gamma
\end{bmatrix}\succeq 0$}, where $\mathbf{\bar{D}}_\tau\triangleq\mathbf{D}_\tau-\diag([L_i\id_n]_{i\in\cN})=\diag([({1\over \tau_i} -L_i)\id_{n}]_{i\in \mathcal{N}})$.
\end{theorem}
Although the proof of Theorem~\ref{thm:general-rate} follows from the lines of \cite{chambolle2015ergodic}, we provide the proof in the appendix for the sake of completeness as it will be used repeatedly to derive our results.

In the following section, we discuss how this algorithm can be implemented to compute an $\epsilon$-optimal solution to \eqref{eq:central_problem} in a distributed way using only $\cO(1/\epsilon)$ communications over the communication graph $\cG$ while respecting node-specific privacy requirements. Later, in Section~\ref{sec:dynamic}, we consider the scenario where the topology of the connectivity graph is time-varying, and propose a distributed algorithm that requires $\cO(1/\epsilon^{1+{1 \over p}})$ communications for any $p\geq 1$. Subsequently, in Section~\ref{sec:dual}, we explain how to extend these rates to distributed resource allocation type problems in \eqref{eq:dual-implement-intro}; and finally, in Section~\ref{sec:numerics} we test the proposed algorithms for solving primal SVM problem in a decentralized manner. In the rest, we assume that the duality gap for \eqref{eq:central_problem} is zero, and a primal-dual solution exists. A sufficient condition for this is the existence of a Slater point, i.e., there exists $\bar{x}\in \mathbf{relint}(\dom \Phi)$ such that $A_i\bar{x}-b_i\in \mathbf{int}(\cK_i)$ for $i\in\cN$, where $\dom\Phi=\cap_{i\in\cN}\dom\Phi_i$.
\noindent \section{Static Network Topology}
\label{sec:static}
\medskip
Let $x_i\in\reals^n$ denote the \emph{local} decision vector of node $i\in\cN$. By taking advantage of the fact that $\cG$ is {\it connected}, we can reformulate \eqref{eq:central_problem} as the following {\it distributed} optimization problem:
{\small
\begin{align}\label{eq:dist_problem}
\min_{\substack{x_i\in\reals^n,\ i\in\cN}} \left\{\sum_{i\in \mathcal{N}}{\Phi_i(x_i)}:x_i=x_j:~\lambda_{ij},\ \forall (i,j)\in \mathcal{E} ,\  A_ix_i-b_i \in \mathcal{K}_i:\ \theta_i,\ \forall{i}\in\mathcal{N} \right\},
\end{align}
}%
where $\lambda_{ij}\in\reals^n$ and $\theta_i\in\reals^{m_i}$ are the corresponding dual variables. Let $\bx=[x_i]_{i\in\cN}\in\reals^{n|\cN|}$. The consensus constraints $x_i=x_j$ for $(i,j)\in \mathcal{E}$ can be formulated as $M{\bf x}=0$, where $M\in \mathbb{R}^{n|\mathcal{E}|\times n|\mathcal{N}|}$ is a block matrix such that $M=H\otimes \id_n$ where $H$ is the oriented edge-node incidence matrix, i.e., the entry $H_{(i,j),l}$, corresponding to edge $(i,j)\in \mathcal{E}$ and $l\in \mathcal{N}$, is equal to $1$ if $l=i$, $-1$ if $l=j$, and $0$ otherwise.
Note that $M^\T M=H^\T H\otimes \id_n=\Omega\otimes \id_n$, where $\Omega\in \mathbb{R}^{|\mathcal{N}|\times |\mathcal{N}|}$ denotes the graph Laplacian of $\cG$, i.e., $\Omega_{ii}=d_i$, $\Omega_{ij}=-1$ if $(i,j)\in\cE$ or $(j,i)\in\cE$, and equal to $0$ otherwise. 

For any closed convex set $\cS$, we have $\sigma^*_{\cS}(\cdot)=\ind{S}(\cdot)$; therefore, using the fact that $\sigma^*_{\cK_i}=\ind{\cK_i}$ for $i\in\cN$,
one can obtain the following saddle point problem corresponding to \eqref{eq:dist_problem},
{\small
\begin{align}\label{eq:static-saddle}
\min_{{\bf x}}\max_{\by}\mathcal{L}({\bf x}, \by)\triangleq\sum_{i\in\mathcal{N}}\bigg(\Phi_i(x_i)+\langle \theta_i,A_ix_i-b_i \rangle -\sigma_{\mathcal{K}_i}(\theta_i) \bigg)+\langle \blambda, M{\bf x}\rangle ,
\end{align}
}%
where $\by=[\btheta^\top~\blambda^\top]^\top$ for $\blambda=[\lambda_{ij}]_{(i,j)\in\mathcal{E}}\in\reals^{n|\cE|}$, $\btheta=[\theta_i]_{i\in \mathcal{N}}\in\reals^m$, 
and $m\triangleq\sum_{i\in\cN}{m_i}$. 

Next, we study the distributed implementation of PDA in \eqref{eq:PDA-x}-\eqref{eq:PDA-y} to solve \eqref{eq:static-saddle}. Let $\Phi(\bx)\triangleq\sum_{i\in\cN}\Phi_i(x_i)$, 
and $h(\by)\triangleq\sum_{i\in\cN}\sigma_{\mathcal{K}_i}(\theta_i)+\fprod{b_i,\theta_i}$. Define the block-diagonal matrix $A\triangleq\diag([A_i]_{i\in\mathcal{N}})\in\reals^{m \times n|\cN|}$ and $T=[A^\top M^\top]^\top$. Therefore, given the initial iterates $\bx^0,\btheta^0,\blambda^0$ and parameters $\gamma>0$, $\tau_i,\kappa_i>0$ for $i\in\cN$, choosing $D_x$ and $D_y$ as defined in Definition~\ref{def:bregman}, and setting $\nu_x=\nu_y=1$, PDA iterations in \eqref{eq:PDA-x}-\eqref{eq:PDA-y} take the following form:   
{\small
\begin{align}\label{eq:pock-pd}
{\bf x}^{k+1}&\gets\argmin_{{\bf x}} \langle \blambda^k, M{\bf x} \rangle + \sum_{i\in\mathcal{N}}\bigg[ \rho_i(x_i)+f_i(x_i^k)+\langle \nabla f(x_i^k),x_i-x_i^k\rangle +\langle A_ix_i-b_i,\theta_i^k\rangle+{1\over 2\tau_i}\|x_i-x_i^k\|^2\bigg], \ \ \ \ \nonumber \\
\theta_i^{k+1}&\gets\argmin_{\theta_i}\sigma_{\mathcal{K}_i}(\theta_i) -\langle A_i(2x_i^{k+1}-x_i^k)-b_i,\theta_i\rangle +{1\over 2\kappa_i}\|\theta_i-\theta_i^k\|^2, \ \ \ \ i\in\mathcal{N} \\
\lambda^{k+1}&\gets\argmin_{\blambda} \left\{-\langle M(2{\bf x}^{k+1}-{\bf x}^k),\blambda \rangle +{1\over 2\gamma}\|\blambda-\blambda^k\|^2\right\}=\blambda^k+\gamma M(2{\bf x}^{k+1}-{\bf x}^k) \nonumber
\end{align}
}%
Since $\cK_i$ is a cone, $\prox{\kappa_i\sigma_{\cK_i}}(.)=\cP_{\cK_i^\circ}(.)$; hence, $\theta_i^{k+1}$ can be written in closed form as   
{\small
\begin{align}
\theta_i^{k+1}\gets\cP_{\cK_i^\circ}\bigg(\theta_i^k+\kappa_i\big(A_i(2x_i^{k+1}-x_i^k)-b_i\big)\bigg), \ \ \ \ i\in\mathcal{N}. \nonumber
\end{align}
}%
Using recursion in $\blambda$ update rule in~\eqref{eq:pock-pd}, we can write $\blambda^{k+1}$ as a partial summation of previous primal variable ${\bf x}^k$ iterates, i.e., $\blambda^k=\blambda^0+\gamma\sum_{\ell=0}^{k-1} M(2{\bf x}^{\ell+1}-{\bf x}^\ell)$. Let $\blambda^0\gets\gamma M\bx^0$, $\bs^0\gets\bx^0$, and $\bs^k\triangleq {\bf x}^k+\sum_{\ell=1}^k {\bf x}^\ell$ for $k\geq 1$; since $M^\top M=\Omega\otimes \id_n$ we obtain that
{\small
\[\langle M{\bf x},\blambda^k \rangle=\gamma \langle {\bf x},~(\Omega \otimes \id_n)\bs^k\rangle=\gamma\sum_{i\in\mathcal{N}}\langle { x_i},~\sum_{j\in\mathcal{N}_i}(s^k_i-s^k_j)\rangle.\]
}%
Thus, PDA iterations given in~\eqref{eq:pock-pd} for the static graph $\cG$ can be computed in \emph{decentralized} way, via the node-specific computations as in Algorithm~DPDA-S displayed in Fig.~\ref{alg:PDS} below.
\begin{figure}[htpb]
\centering
\framebox{\parbox{\columnwidth}{
{\small
\textbf{Algorithm DPDA-S} ( $\bx^{0},\btheta^0,\gamma,\{\tau_i,\kappa_i\}_{i\in\cN}$ ) \\[1.5mm]
Initialization: $s_i^0\gets x_i^0$, \quad $i\in\cN$\\
Step $k$: ($k \geq 0$)\\
\text{ } 1. $x_i^{k+1}\gets\prox{\tau_i\rho_i}\left(x_i^k-\tau_i\bigg(\nabla f_i(x_i^k)+A_i^\top\theta^k_i+\gamma \sum_{j\in\mathcal{N}_i}(s^k_i-s^k_j)\bigg)\right), \quad i\in\cN$  \\
\text{ } 2. $s_i^{k+1}\gets x_i^{k+1}+\sum_{\ell=1}^{k+1} x_i^\ell, \quad i \in \cN$\\
\text{ } 3. $\theta_i^{k+1}\gets\cP_{\cK_i^\circ}\bigg(\theta_i^k+\kappa_i\big(A_i(2x_i^{k+1}-x_i^k)-b_i\big)\bigg), \quad i \in \cN$\\
 }}}
\caption{\small Distributed Primal Dual Algorithm for Static $\cG$ (DPDA-S)}
\label{alg:PDS}
\vspace*{-3mm}
\end{figure}

The convergence rate for DPDA-S, given in~\eqref{eq:DPA-rate}, follows from Theorem~\ref{thm:general-rate}
with the help of following technical lemma.
\begin{lemma}\label{lem:schur}
Given $\{\tau_i,\kappa_i\}_{i\in\mathcal{N}}$ and $\gamma$ such that $\gamma>0$, and $\tau_i>0$, $\kappa_i>0$ for $i\in\cN$, let 
$A_0=M$ and $A\triangleq\diag([A_i]_{i\in\mathcal{N}})$. Then $\mathbf{\bar{Q}}\triangleq\mathbf{\bar{Q}}(A,A_0)\succeq 0$ if $\{\tau_i,\kappa_i\}_{i\in\mathcal{N}}$ and $\gamma$ are chosen such that
{\small
\begin{align}\label{eq:schur-cond}
\left({1\over \tau_i}-L_i-2\gamma d_i\right){1\over \kappa_i} \geq \sigma^2_{\max}(A_i),\quad \forall~i\in\mathcal{N},
\end{align}
}%
where $\mathbf{\bar{Q}}(A,A_0)$ is defined in Theorem~\ref{thm:general-rate}.
\end{lemma}
\begin{proof}
Since $\mathbf{D}_\gamma\succ 0$, Schur complement condition implies that $\mathbf{Q}\succeq 0$ if and only if
{\small
\begin{align}\label{eq:schur-cond-2}
B-\gamma \begin{bmatrix} M^\top M & 0 \\ 0 & 0\end{bmatrix} \succeq 0,\quad \hbox{where}\quad B\triangleq\begin{bmatrix}\mathbf{\bar{D}}_\tau & -A^\top\\ -A & \mathbf{D}_\kappa \end{bmatrix}.
\end{align}}%
Moreover, since $\mathbf{D}_\kappa\succ 0$, again using Schur complement and the fact that $M^\top M=\Omega\otimes \id_n$, one can conclude that \eqref{eq:schur-cond-2} holds if and only if $\mathbf{\bar{D}}_\tau-\gamma(\Omega\otimes \id_n)-A^\top \mathbf{D}_\kappa^{-1}A\succeq 0$. By definition $\Omega=\diag([d_i]_{i\in\mathcal{N}})-E$, where $E_{ii}=0$ for all $i\in \mathcal{N}$ and $E_{ij}=E_{ji}=1$ if $(i,j)\in\mathcal{E}$ or $(j,i)\in \mathcal{E}$. Note that $\diag([d_i]_{i\in\mathcal{N}})+E \succeq 0$ since it is diagonally dominant. Therefore, $\Omega \preceq 2\diag([d_i]_{i\in\mathcal{N}})$. Hence, it is sufficient to have $\mathbf{\bar{D}}_\tau- 2\gamma\diag([d_i]_{i\in\mathcal{N}})\otimes \id_n-A^\top \mathbf{D}_\kappa^{-1}A\succeq 0$, and this condition holds if the condition in the statement of the lemma is true.
\end{proof}
\begin{remark}
Note that for all $i\in\cN$ by choosing $\tau_i= {1\over c_i+L_i+2\gamma d_i}$, $\kappa_i={c_i\over \sigma^2_{\max}(A_i)}$ for any $c_i>0$, the condition in Lemma~\ref{lem:schur} is satisfied.
\end{remark}
Next, we refine the error bound in~\eqref{eq:DPA-rate}, and quantify the suboptimality and infeasibility of the DPDA-S iterate sequence. 
\begin{theorem}\label{thm:static-error-bounds}
Let $({\bf x}^*,\btheta^*,\blambda^*)$ be an arbitrary saddle-point for $\cL$ in \eqref{eq:static-saddle}, and $\{\bx^k\}_{k\geq 0}$ be the iterate sequence generated using Algorithm DPDA-S, displayed in Fig.~\ref{alg:PDS}, initialized from an arbitrary $\bx^0$ and $\btheta^0=\mathbf{0}$. Let primal-dual step-sizes $\{\tau_i,\kappa_i\}_{i\in\cN}$ and $\gamma$ be chosen such that the condition \eqref{eq:schur-cond} in Lemma~\ref{lem:schur} holds.
Then, $\{\bar{\bx}^K\}$ has a limit point and every limit point has the form $\one_{|\cN|}\otimes x^*$ such that $x^*$ is an optimal solution to \eqref{eq:central_problem}. In particular, the following error bounds hold for all $K\geq 1$:
{\small
\begin{equation*}
\norm{\blambda^*}\|M\bar{\bf x}^K\|+\sum_{i\in\cN}\norm{\theta_i^*} d_{\mathcal{K}_i}(A_i\bar{\bf x}_i^K-b_i)\leq {\Theta_1\over K},\qquad |\Phi(\bar{\bx}^K)-\Phi({\bf x}^*)|\leq {\Theta_1\over K},
\end{equation*}
}%
where $\Theta_1\triangleq {2\over \gamma}\|\blambda^*\|^2-\frac{\gamma}{2}\norm{M\bx^0}^2+\sum_{i\in\mathcal{N}}\bigg[{1\over 2\tau_i}\|x^*_i-x_i^0\|^2+{4\over \kappa_i}\|\theta^*_i\|^2\bigg]$,
and $\bar{\bf x}^K\triangleq {1\over K}\sum_{k=1}^K\bx^k$.
\end{theorem}
\begin{proof}
We start the proof with a simple observation. Every closed convex cone $\cC\in\reals^m$ induces a decomposition on $\reals^{m}$, i.e., according to Moreau decomposition, for any $y\in\reals^m$, there exist $y^1,y^2\in\reals^m$ such that $y^1\perp y^2$ and $y=y^1+y^2$; in particular, $y^1=\cP_{\mathcal{\cC}}(y)$ and $y^2=\cP_{\cC^\circ}(y)$ where $\mathcal{\cC}^\circ=-\cC^*$ is the polar cone of $\cC$. Hence, it follows from the definition of a support function and the fact that $\langle y, w\rangle \leq 0$ for any $y\in\cC^{\circ}$ and $w\in \cC$, one can conclude that
{\small
\begin{align}\label{eq:polar-support}
\sigma_{\cC}(y)=
\begin{cases}
0 & y\in \cC^{\circ} \\
+\infty & \hbox{o.w.}
\end{cases}
\end{align}
}%
Note the iterate sequence $\{\bx^k,\btheta^k\}_{k\geq 0}$ generated by Algorithm DPDA-S in Fig.~\ref{alg:PDS} is the same as the PDA iterate sequence $\{{\bf x}^k,\btheta^k,\blambda^k\}_{k\geq 0}$ computed according to \eqref{eq:pock-pd} for solving \eqref{eq:static-saddle} when $\blambda^0=\gamma M{\bf x}^0$. Since the step-size parameters $\{\tau_i,\kappa_i\}_{i\in\cN}$ and $\gamma$ are chosen satisfying the condition \eqref{eq:schur-cond} in Lemma~\ref{lem:schur}, the requirement $\mathbf{\bar{Q}}(A,A_0)\succeq 0$ in Theorem~\ref{thm:general-rate} is true, where $A_0=M$ for problem \eqref{eq:static-saddle}. Therefore, Theorem~\ref{thm:general-rate} implies that \eqref{eq:DPA-rate} holds for all $K\geq 1$ with $\nu_x=\nu_y=1$ and Bregman function $D_x$, $D_y$ defined as in Definition~\ref{def:bregman}. In particular, the result of Theorem~\ref{thm:general-rate} can be written more explicitly for \eqref{eq:static-saddle} as follows: let $\bar{\bf x}^K\triangleq {1\over K}\sum_{k=1}^K\bx_k$, $\bar{\btheta}^K\triangleq {1\over K}\sum_{k=1}^K\btheta^k$ and $\bar{\blambda}^K\triangleq {1\over K}\sum_{k=1}^K\blambda^k$, then for any ${\bf x}\in \mathbb{R}^{n|\mathcal{N}|}$, $\blambda\in \mathbb{R}^{n|\mathcal{E}|}$, $\btheta\in \mathbb{R}^m$ for $m=\sum_{i\in\cN}m_i$, and for all $K\geq 1$, we have
{\small
\begin{align}\label{eq:saddle-rate}
\mathcal{L}(\bar{\bf x}^K,\btheta,\blambda)-&\mathcal{L}({\bf x},\bar{\btheta}^K,\bar{\blambda}^K)\leq \Theta(\bx,\btheta,\blambda)/K,\\
\Theta(\bx,\btheta,\blambda)\triangleq&{1\over 2\gamma}\|\blambda-\blambda^0\|^2-\langle M({\bf x}-{\bf x}^0),\blambda-\blambda^0\rangle \nonumber \\
&+\sum_{i\in\mathcal{N}}\bigg[{1\over 2\tau_i}\|x_i-x_i^0\|^2+{1\over 2\kappa_i}\|\theta_i-\theta_i^0\|^2-\langle A_i(x_i-x_i^0),\theta_i-\theta_i^0\rangle \bigg].\nonumber
\end{align}
}%
Note that under the assumption in \eqref{eq:schur-cond}, Schur complement condition guarantees that
{\small
\begin{equation*}
\begin{bmatrix}
\frac{1}{\tau_i}\id_n &-A_i^\top\\
-A_i & \frac{1}{\kappa_i}\id_{m_i} \\
\end{bmatrix}
\preceq
\begin{bmatrix}
\frac{2}{\tau_i}\id_n & \mathbf{0}^\top\\
\mathbf{0} & \frac{2}{\kappa_i}\id_{m_i} \\
\end{bmatrix}.
\end{equation*}}%
Therefore,
{\small
\begin{eqnarray}
\label{eq:simple-C-bound}
\Theta(\bx,\btheta,\blambda)\leq \sum_{i\in\mathcal{N}}\bigg[{1\over\tau_i}\|x_i-x_i^0\|^2+{1\over \kappa_i}\|\theta_i-\theta_i^0\|^2\bigg]+{1\over 2\gamma}\|\blambda-\blambda^0\|^2-\langle M({\bf x}-{\bf x}^0),\blambda-\blambda^0\rangle.
\end{eqnarray}
}%
Let $({\bf x}^*,\btheta^*,\blambda^*)$ be an arbitrary saddle-point for $\cL$ in \eqref{eq:static-saddle}; hence, $\mathcal{L}( \bx^*,\btheta^*,\blambda^*)=\Phi(\bx^*)$ and $\theta_i^*\in\cK_i^\circ$ for $i\in\cN$.
Define $\tilde{\bw}=[\tilde{w}_i]_{i\in\cN}$ such that $\tilde{w}_i\triangleq A_i\bar{x}_i^K-b_i\in\reals^{m_i}$ for $i\in\cN$. Since $\cK_i$ is a closed convex cone, it induces a decomposition on $\reals^{m_i}$ for $i\in\cN$, i.e., consider $\tilde{w}_i^1=\cP_{\mathcal{K}_i}(\tilde{w}_i)$ and $\tilde{w}_i^2=\cP_{\mathcal{K}_i^\circ}(\tilde{w}_i)$. Note that since $\tilde{w}_i=\tilde{w}_i^1+\tilde{w}_i^2$, $\norm{\tilde{w}_i^2}=\norm{\cP_{\cK_i}(\tilde{w}_i)-\tilde{w}_i}=d_{\cK_i}(\tilde{w}_i)$. Define $\tilde{\btheta}=[\tilde{\theta}_i]_{i\in\cN}$ such that $\tilde{\theta}_i\triangleq 2\|\theta_i^*\|{ 1 \over \|\tilde{w}_i^2\|}\tilde{w}_i^2\in\cK_i^\circ$. Therefore,
{\small
\begin{equation}
\label{eq:tilde-theta}
\langle A_i\bar{ x}_i^K -b_i, \tilde{\theta}_i \rangle=2\frac{\|\theta_i^*\|}{\|\tilde{w}_i^2\|}\fprod{\tilde{w}_i^1+\tilde{w}_i^2,~\tilde{w}_i^2}=2\|\theta_i^*\| d_{\mathcal{K}_i}(A_i\bar{ x}_i^K -b_i),
\end{equation}}%
where the second equality follows from $\tilde{w}_i^1\perp \tilde{w}_i^2$. Similarly, define $\tilde{\blambda}\triangleq 2\|\blambda^*\|(M\bar{\bf x}^K/\norm{M\bar{\bf x}^K})$. Hence, $\langle M\bar{\bf x}^K, \tilde{\blambda}\rangle=2\|\blambda^*\| \|M\bar{\bf x}^K\|$. Therefore, together with \eqref{eq:tilde-theta}, we get
{\small
\begin{align*}
\mathcal{L}(\bar{\bf x}^K,\tilde{\btheta},\tilde{\blambda})-\mathcal{L}(\bx^*,\btheta^*,\blambda^*)= \Phi(\bar{\bx}^K)-\Phi({\bx}^*)+2\left(\|\blambda^*\| \|M\bar{\bx}^K\|+\sum_{i\in\cN} d_{\mathcal{K}_i}(A_i\bar{x}_i^K-b_i)\|\theta_i^*\|\right).
\end{align*}}%
Now we are going to upper bound $\Theta(\bx^*,\tilde{\btheta},\tilde{\blambda})$ using \eqref{eq:simple-C-bound}. Since $\blambda^0=\gamma M\bx^0$, we get
{\small
\begin{align}
\label{eq:tilde-lambda}
 {1\over 2\gamma}\|\tilde{\blambda}-\blambda^0\|^2-\langle M(\bx^*-{\bf x}^0),\tilde{\blambda}-\blambda^0\rangle
 &={1\over 2\gamma}\left(\|\tilde{\blambda}-\blambda^0-\gamma M(\bx^*-\bx^0)\|^2-\gamma^2\|M(\bx^*-\bx^0)\|^2\right) \nonumber \\
 &={2\over \gamma}\|\blambda^*\|^2-\frac{\gamma}{2}\norm{M\bx^0}^2,
\end{align}}%
where in the last equality follows from $M{\bf x}^*=0$. 
Since $\btheta^*$ and $\blambda^*$ maximize the Lagrangian function at $\bx^*$, and we set $\theta_i^0=\mathbf{0}$, the definitions of $\tilde{\theta}_i$, $\tilde{\blambda}$, and \eqref{eq:saddle-rate}, \eqref{eq:simple-C-bound} together imply that
{\small
\begin{equation*}
\label{eq:aux-upper-1}
\mathcal{L}(\bar{\bf x}^K,\tilde{\btheta},\tilde{\blambda})-\mathcal{L}(\bx^*,\btheta^*,\blambda^*)\leq \mathcal{L}(\bar{\bf x}^K,\tilde{\btheta},\tilde{\blambda})-\mathcal{L}(\bx^*,\bar{\btheta}^K,\bar{\blambda}^K)\leq \frac{1}{K}\Theta(\bx^*,\tilde{\btheta},\tilde{\blambda})\leq\frac{\Theta_1}{K}.
\end{equation*}}%
Therefore, we can conclude that
{\small
\begin{equation}
\label{eq:upper-bound}
\Phi(\bar{\bx}^K)-\Phi({\bx}^*)+2\left(\|\blambda^*\| \|M\bar{\bx}^K\|+\sum_{i\in\cN} d_{\mathcal{K}_i}(A_i\bar{x}_i^K-b_i)\|\theta_i^*\|\right)\leq \frac{\Theta_1}{K},
\end{equation}}%
where we use $\mathcal{L}(\bx^*,\btheta^*,\blambda^*)=\Phi({\bx}^*)$ and the fact that $\sigma_{\mathcal{K}_i}(\tilde{\theta}_i)=0$ due to \eqref{eq:polar-support} since $\tilde{\theta}_i\in\cK_i^\circ$ for $i\in\cN$. Moreover, since $({\bf x}^*,\btheta^*,\blambda^*)$ is a saddle-point for $\cL$ in \eqref{eq:static-saddle}, we clearly have $\mathcal{L}(\bar{\bf x}^K,\btheta^*,\blambda^*)-\mathcal{L}({\bf x}^*,\btheta^*,\blambda^*) \geq 0$; therefore,
{\small
\begin{align}\label{eq:aux-lower}
\Phi(\bar{\bf x}^K)-\Phi({\bf x}^*)+\langle \blambda^*,M\bar{\bf x}^K\rangle+\sum_{i\in\cN}\fprod{\theta_i^*,~A_i\bar{\bf x}_i^K-b_i}\geq 0.
\end{align}}%
Recall that $i\in\cN$ we defined $\tilde{w}_i^1=\cP_{\mathcal{K}_i}(\tilde{w}_i)$ and $\tilde{w}_i^2=\cP_{\mathcal{K}_i^\circ}(\tilde{w}_i)$, where $\tilde{w}_i\triangleq A_i\bar{x}_i^K-b_i\in\reals^{m_i}$.
For all $i\in\cN$, $\theta_i^*\in \mathcal{K}_i^{\circ}$ and $\tilde{w}_i^1\in \mathcal{K}_i$ imply $\langle \theta^*_i,\tilde{w}_i^1 \rangle \leq 0$; hence, for all $i\in\cN$,
{\small
$$\langle A_i\bar{\bf x}_i^K-b_i, \theta_i^*\rangle=\langle \tilde{w}_i-\tilde{w}_i^1+\tilde{w}_i^1,~\theta_i^* \rangle\leq \langle \tilde{w}_i-\tilde{w}_i^1,~\theta_i^* \rangle\leq \|\theta^*_i\| d_{\mathcal{K}_i}(A_i\bar{\bf x}_i^K-b_i).$$}%
Together with \eqref{eq:aux-lower}, we conclude that
{\small
\begin{equation}\label{eq:lower-bound}
\Phi(\bar{\bf x}^K)-\Phi({\bf x}^*)+\|\blambda^*\| \|M\bar{\bf x}^K\|+\sum_{i\in\cN}\|\theta^*_i\| d_{\mathcal{K}_i}(A_i\bar{\bf x}_i^K-b_i) \geq 0.
\end{equation}
}%
By combining inequalities \eqref{eq:upper-bound} and \eqref{eq:lower-bound} immediately implies the desired result.
\end{proof}
\section{Dynamic Network Topology}
\label{sec:dynamic}
In this section we develop a distributed primal-dual algorithm for solving \eqref{eq:central_problem} when the communication network topology is \emph{time-varying}. We assume a \emph{compact} domain, i.e., let $D_i\triangleq\max_{x_i,x'_i\in \dom \rho_i}\|x-x'\|$ and $B\triangleq\max_{i\in\mathcal{N}}D_i<\infty$. Let
{\small
\begin{equation*}
C\triangleq\{{\bf x}\in \mathbb{R}^{n|\mathcal{N}|}:\ \exists \bar{x}\in\mathbb{R}^n\quad \hbox{s.t.}\quad \ x_i=\bar{x},\ \forall i\in\mathcal{N},\quad \|\bar{x}\|\leq B \},
\end{equation*}}%
then one can reformulate \eqref{eq:central_problem} in a decentralized way as follows:
{\small
\begin{align}\label{eq:dynamic-saddle}
\min_{\bf x} \max_{\by} \cL(\bx,\by)\triangleq\sum_{i\in\mathcal{N}}\bigg(\Phi_i(x_i)+\langle \theta_i,A_ix_i-b_i\rangle-\sigma_{\mathcal{K}_i}(\theta_i)\bigg)+\fprod{\blambda,~\bx}-\sigma_C(\blambda),
\end{align}}%
where $\by=[\btheta^\top \blambda^\top]^\top$ such that $\blambda\in\reals^{n|\cN|}$, $\btheta=[\theta_i]_{i\in \mathcal{N}}\in\reals^m$, and $m\triangleq\sum_{i\in\cN}{m_i}$.

Next, we consider the implementation of PDA in \eqref{eq:PDA-x}-\eqref{eq:PDA-y} to solve \eqref{eq:dynamic-saddle}. Let $\Phi(\bx)\triangleq\sum_{i\in\cN}\Phi_i(x_i)$, and $h(\by)\triangleq\sigma_{\cC}(\blambda)+\sum_{i\in\cN}\sigma_{\mathcal{K}_i}(\theta_i)+\fprod{b_i,\theta_i}$. Define the block-diagonal matrix $A\triangleq\diag([A_i]_{i\in\mathcal{N}})\in\reals^{m \times n|\cN|}$ and $T=[A^\top~\id_{n|\cN|}]^\top$. Therefore, given the initial iterates $\bx^0,\btheta^0,\blambda^0$ and parameters $\gamma>0$, $\tau_i,\kappa_i>0$ for $i\in\cN$, choosing $D_x$ and $D_y$ as defined in Definition~\ref{def:bregman}, and setting $\nu_x=\nu_y=1$, PDA iterations given in \eqref{eq:PDA-x}-\eqref{eq:PDA-y} take the following form: Starting from $\bmu^0=\blambda^0$, compute for $i\in\cN$
{\small
\begin{align}\label{eq:pock-pd-2}
&x_i^{k+1}\gets\argmin_{{\bf x}} \rho_i(x_i)+f_i(x_i^k)+\langle \nabla f(x_i^k),x_i-x_i^k\rangle +\langle A_ix_i-b_i,\theta_i^k\rangle+\langle x_i,\mu_i^k \rangle +{1\over 2\tau_i}\|x_i-x_i^k\|^2_2,\nonumber \\
&\theta_i^{k+1}\gets\argmin_{\theta_i}\sigma_{\mathcal{K}_i}(\theta_i) -\langle A_i(2x_i^{k+1}-x_i^k)-b_i,~\theta_i\rangle +{1\over 2\kappa_i}\|\theta_i-\theta_i^k\|_2^2,\\
&\blambda^{k+1}\gets\argmin_{\bmu} \sigma_C(\bmu)-\langle 2{\bf x}^{k+1}-{\bf x}^k,\bmu \rangle +{1\over 2\gamma}\|\bmu-\bmu^k\|_2^2,\qquad \bmu^{k+1}\gets\blambda^{k+1}. \nonumber
\end{align}
}%
Using extended Moreau decomposition for proximal operators, $\blambda^{k+1}$ can be written as
{\small
\begin{align}\label{eq:extended-Moreau}
 \blambda^{k+1}&=\argmin_{\bmu}\sigma_C(\bmu)+{1\over 2\gamma}\|\bmu-(\bmu^k+\gamma(2{\bf x}^{k+1}-{\bf x}^k))\|^2 =\prox{\gamma\sigma_C}(\bmu^k+\gamma(2{\bf x}^{k+1}-{\bf x}^k)) \nonumber \\
 &=\bmu^k+\gamma(2{\bf x}^{k+1}-{\bf x}^k)-\gamma~\mathcal{P}_C\left( \frac{1}{\gamma}\bmu^k+2{\bf x}^{k+1}-{\bf x}^k\right).
\end{align}
}%
Let $\one\in \mathbb{R}^{|\mathcal{N}|}$ be the vector all ones, and $\cB_0\triangleq\{x\in\reals^n:\ \norm{x}\leq B\}$. Note that $\cP_{\cB_0}(x)=x\min\{1,\frac{B}{\norm{x}}\}$. For any $\bx=[x_i]_{i\in\cN}\in \mathbb{R}^{n|\mathcal{N}|}$, $\mathcal{P}_C(\bx)$ can be computed as
{\small
\begin{equation}
\mathcal{P}_C(\bx)=\one\otimes p(\bx),\quad\hbox{\normalsize where}\quad p(\bx)\triangleq \argmin_{\xi\in\cB_0} \sum_{i\in\mathcal{N}}\|\xi-x_i\|^2=\argmin_{\xi\in\cB_0}\|\xi-{1\over |\mathcal{N}|}\sum_{i\in\mathcal{N}}x_i\|^2.
\end{equation}
}%
Let $\mathcal{B}\triangleq\{\bx: \ \|x_i\|\leq B,\ i\in\cN\}=\Pi_{i\in\cN}\cB_0$. Hence, we can equivalently write
{\small
\begin{equation}
\label{eq:proj}
\mathcal{P}_C(\bx)=\mathcal{P}_{\mathcal{B}}\left((W\otimes\id_n)\bx\right),\quad \hbox{\normalsize where}\quad W\triangleq\frac{1}{|\cN|}\one\one^\top\in\reals^{|\cN|\times|\cN|}.
\end{equation}
}%
Equivalently,
{\small
\begin{equation}
\label{eq:proj-2}
\mathcal{P}_C(\bx)=\mathcal{P}_{\mathcal{B}}\left(\one\otimes\tilde{p}(\bx)\right),\quad \hbox{\normalsize where}\quad \tilde{p}(\bx)\triangleq{1\over |\mathcal{N}|}\sum_{i\in\mathcal{N}}x_i.
\end{equation}
}%
Although $\bx$-step and $\btheta$-step of the PDA implementation in \eqref{eq:pock-pd-2} can be computed locally at each node, computing $\blambda^{k+1}$ requires communication among the nodes. Indeed, evaluating the average operator $\tilde{p}(.)$ is \emph{not} a simple operation in a decentralized computational setting which only allows for communication among neighbors. In order to overcome this issue, we will approximate $\tilde{p}(.)$ operator using multi-consensus steps, and analyze the resulting iterations as an inexact primal-dual algorithm. In~\cite{chen2012fast}, this idea has been exploited within a distributed primal algorithm for unconstrained consensus optimization problems. We define the \emph{consensus step} as one time exchanging local variables among neighboring nodes -- the details of this operation will be discussed shortly. 
Since the connectivity network is dynamic, let $\cG^t=(\cN,\cE^t)$ be the connectivity network when consensus step-$t$ is realized for $t\in\integers_+$. We adopt the information exchange model 
in \cite{nedic2009distributed}.
\begin{assumption}\label{assum1}
Let $V^t\in\reals^{|\cN|\times|\cN|}$ be the weight matrix for consensus step-$t$ corresponding to $\cG^t=(\cN,\cE^t)$. 
Suppose $V^t$ satisfies following conditions for all $t\in\integers_+$: (\textbf{i}) $V^t$ is a doubly stochastic matrix; (\textbf{ii}) there exists $\zeta\in(0,1)$ such that for any $i\in\mathcal{N}$, $V^t_{ij}\geq \zeta$ if $j\in\mathcal{N}_i$, and $V^t_{ij}=0$ if $j\notin\mathcal{N}_i$; (\textbf{iii}) $\cE^{\infty}\triangleq\{(i,j)\in\cN\times\cN: (i,j)\in\cE^t\ \hbox{ for infinitely many }\ t\in\integers_+\}$ is connected, and there exists an integer $T>1$ such that if $(i,j)\in \cE^{\infty}$, then $(i,j)\in \cE^t\cup \cE^{t+1}\cup ...\cup \cE^{t+T-1}$ for all $t\geq 1$.
\end{assumption}
\begin{lemma}\label{lem:approximation}
~\cite{nedic2009distributed}~Let Assumption \ref{assum1} holds, and $W^{t,s}=V^tV^{t-1}...V^{s+1}$ for $t\geq s+1$. Given $s\geq 0$ the entries of $W^{t,s}$ converges to ${1\over N}$ as $t\rightarrow \infty$ with a geometric rate. In particular, for all $i,j\in\mathcal{N}$, one has $\left|W^{t,s}_{ij} -{1\over N}\right|\leq \Gamma \alpha^{t-s}$,
where $\Gamma=2(1+\zeta^{-\bar{T}})/(1-\zeta^{\bar{T}})$, $\alpha=(1-\zeta^{\bar{T}})^{1\over \bar{T}}$, and $\bar{T}=(N-1)T$.
\end{lemma}
Consider the $k$-th iteration of PDA as shown in \eqref{eq:pock-pd-2}. Instead of computing $\blambda^{k+1}$ exactly according to \eqref{eq:extended-Moreau}, we propose to approximate $\blambda^{k+1}$ with the help of Lemma~\ref{lem:approximation} and set $\bmu^{k+1}$ to this approximation. In particular, let $t_k$ be the total number of consensus steps done before $k$-th iteration of PDA, 
and let $q_k\geq 1$ be the number of consensus steps within iteration $k$. For $\bx=[x_i]_{i\in\cN}$, define
{\small
\begin{equation}
\label{eq:approx-average}
\cR^k(\bx)\triangleq\mathcal{P}_{\mathcal{B}}\left((W^{t_k+q_k,t_k}\otimes\id_n)~\bx\right)
\end{equation}
}%
to approximate $\mathcal{P}_C(\bx)$ in \eqref{eq:extended-Moreau}. Note that $\mathcal{R}^k(\cdot)$ can be computed in a distributed fashion requiring $q_k$ communications with the neighbors for each node. 
In particular, components of $\cR^k(\bx)$ can be computed at each node as follows:
{\small
\begin{equation}
\label{eq:component-approx}
\cR^k(\bx)=[\cR_i^k(\bx)]_{i\in\cN}\quad \hbox{\normalsize such that}\quad \cR_i^k(\bx)\triangleq\cP_{\cB_0}(\sum_{j\in\cN}W^{t_k+q_k,t_k}_{ij}x_j).
\end{equation}
}%
Moreover, the approximation error, 
$\mathcal{R}^k(\bx)-\mathcal{P}_C(\bx)$, for any $\bx$ can be bounded as in \eqref{eq:approx_error} due to non-expansive property of $\cP_{\cB}$ and using Lemma~\ref{lem:approximation}. 
From \eqref{eq:proj-2}, we get for all $i\in\cN$,
{\small
\begin{align}
\|\mathcal{R}_i^k(\bx)-\cP_{\cB_0}\big(\tilde{p}(\bx)\big)\|
&=\|\cP_{\cB_0}\big(\sum_{j\in\cN}W^{t_k+q_k,t_k}_{ij}x_j\big)-\cP_{\cB_0}\big(\tfrac{1}{N}\sum_{j\in\mathcal{N}}x_j\big)\| \nonumber\\
&\leq\|\sum_{j\in\cN}\big(W_{ij}^{t_k+q_k,t_k}-\tfrac{1}{N}\big)x_j\|\leq \sqrt{N}~\Gamma \alpha^{q_k}\norm{\bx}. \label{eq:approx_error}
\end{align}
}%
Thus, \eqref{eq:proj-2} implies that $\|\mathcal{R}^k(\bx)-\mathcal{P}_C(\bx)\| \leq N~\Gamma \alpha^{q_k}\norm{\bx}$.
Now, using \eqref{eq:approx-average}, we will approximate $\blambda^{k+1}$ computation in \eqref{eq:extended-Moreau} with the following update rule:
{\small
\begin{equation}\label{eq:inexact-rule}
{\bmu}^{k+1}\gets{\bmu}^k+\gamma(2{\bf x}^{k+1}-{\bf x}^k)-\gamma\cR^k\left(\tfrac{1}{\gamma}{\bmu}^k+2{\bf x}^{k+1}-{\bf x}^k\right).
\end{equation}
}%
and replace the exact computation in \eqref{eq:pock-pd-2}, $\bmu^{k+1}\gets\blambda^{k+1}$, with the inexact iteration rule in \eqref{eq:inexact-rule}. Thus, PDA iterations given in~\eqref{eq:pock-pd-2} can be computed inexactly, but in \emph{decentralized} way for dynamic connectivity, via the node-specific computations as in Algorithm~DPDA-D displayed in Fig.~\ref{alg:PDD} below.
\begin{figure}[htpb]
\centering
\framebox{\parbox{\columnwidth}{
{\small
\textbf{Algorithm DPDA-D} ( $\bx^{0},\btheta^0,\gamma,\{\tau_i,\kappa_i\}_{i\in\cN}$ ) \\[1.5mm]
Initialization: $\mu_i^0\gets\mathbf{0}$, \quad $i\in\cN$\\
Step $k$: ($k \geq 0$)\\
\text{ } 1. $x_i^{k+1}\gets\prox{\tau_i\rho_i}\left(x_i^k-\tau_i\bigg(\nabla f_i(x_i^k)+A_i^\top\theta^k_i+\mu_i^k\bigg)\right), \quad i\in\cN$  \\
\text{ } 2. ${\mu}_i^{k+1}\gets{\mu}_i^k+\gamma(2x_i^{k+1}-x_i^k)-\gamma\cR_i^k\left(\tfrac{1}{\gamma}{\bmu}^k+2{\bf x}^{k+1}-{\bf x}^k\right), \quad i \in \cN$\\
\text{ } 3. $\theta_i^{k+1}\gets\cP_{\cK_i^\circ}\bigg(\theta_i^k+\kappa_i\big(A_i(2x_i^{k+1}-x_i^k)-b_i\big)\bigg), \quad i \in \cN$\\
 }}}
\caption{\small Distributed Primal Dual Algorithm for Dynamic $\cG^t$ (DPDA-D)}
\label{alg:PDD}
\vspace*{-5mm}
\end{figure}

Next, we define the proximal error sequence $\{\be^k\}_{k\geq 1}$ as in \eqref{eq:prox-error-seq}, which will be used later for analyzing the convergence of Algorithm~DPDA-D displayed in Fig.~\ref{alg:PDD}.
{\small
\begin{equation}
\label{eq:prox-error-seq}
\be^{k+1} \triangleq \cP_{\cC}\left(\tfrac{1}{\gamma}{\bmu}^k+2{\bf x}^{k+1}-{\bf x}^k\right)-\cR^k\left(\tfrac{1}{\gamma}{\bmu}^k+2{\bf x}^{k+1}-{\bf x}^k\right);
\end{equation}
}%
hence, $\bmu^{k}=\blambda^{k}+\gamma \be^{k}$ for $k\geq 1$. In the rest, we assume that $\bmu^0=\mathbf{0}$. The following observation will also be useful to prove error bounds for DPDA-D iterate sequence. For each $i\in\cN$, the definition of $\cR_i^k$ in \eqref{eq:component-approx} implies that $\cR_i^k(\bx)\in\cB_0$ for all $\bx$; hence, from \eqref{eq:inexact-rule},
{\small
\begin{eqnarray*}
\|\mu_i^{k+1}\| \leq \|\mu_i^k+\gamma(2x_i^{k+1}-x_i^k)\|+\gamma\|\cR_i^k\big(\tfrac{1}{\gamma}{\mu}^k
+2{\bf x}^{k+1}-{\bf x}^k\big)\| \leq \|\mu_i^k\|+4\gamma B.
\end{eqnarray*}
}%
Thus, we trivially get the following bound on $\norm{\bmu^k}$:
{\small
\begin{equation}\label{eq:mu-bound}
\|\bmu^k\| \leq 4\gamma \sqrt{N}~B~k.   
\end{equation}
}%
Moreover, for any $\bmu$ and $\blambda$ we have that
{\small
\begin{equation} \label{eq:support-error}
\sigma_C({\bmu})=\sup_{{\bf x}\in C}\langle \blambda,{\bf x} \rangle+\langle \bmu-\blambda,{\bf x} \rangle \leq \sigma_C(\blambda)+\sqrt{N}~B~\|\bmu-\blambda\|.
\end{equation}
}%
\begin{theorem}\label{thm:dynamic-error-bounds}
Let $({\bf x}^*,\btheta^*,\blambda^*)$ be an arbitrary saddle-point for $\cL$ in \eqref{eq:dynamic-saddle}. Starting from $\bmu^0=\mathbf{0}$, $\btheta^0=\mathbf{0}$, and an arbitrary $\bx^0$, let $\{\bx^k,\bmu^k,\btheta^k\}_{k\geq 0}$ be the iterate sequence generated using Algorithm DPDA-D, displayed in Fig.~\ref{alg:PDD}, using $q_k=\sqrt[\leftroot{-3}\uproot{3}p]{k}$ consensus steps at the $k$-th iteration for all $k\geq 1$ for some rational $p\geq 1$. Let primal-dual step-sizes $\{\tau_i,\kappa_i\}_{i\in\cN}$ and $\gamma$ be chosen such that \eqref{eq:step-rule-dynamic} holds.
{\small
\begin{equation}\label{eq:step-rule-dynamic}
\Big(\frac{1}{\tau_i}-L_i-\gamma\Big)\frac{1}{\kappa_i} \geq \sigma^2_{\max}(A_i),\quad \forall\ i\in\mathcal{N}.
\end{equation}
}%
Then, the sequence $\{\bar{\bf x}^K\}$, where $\bar{\bf x}^K\triangleq {1\over K}\sum_{k=1}^K\bx^k$, has a limit point and every limit point has the form $\one_{|\cN|}\otimes x^*$ such that $x^*$ is an optimal solution to \eqref{eq:central_problem}. In particular,   
{\small
\begin{equation*}
\norm{\blambda^*}d_{\tilde{C}}(\bar{\bx}^K)+\sum_{i\in\cN}\norm{\theta_i^*} d_{\mathcal{K}_i}(A_i\bar{\bf x}_i^K-b_i)\leq \frac{\Theta_2+\Theta_3(K)}{K},\qquad |\Phi(\bar{\bx}^K)-\Phi({\bf x}^*)|\leq \frac{\Theta_2+\Theta_3(K)}{K},
\end{equation*}
}%
where 
$\Theta_2\triangleq 2\|\blambda^*\|\big({1\over \gamma}\|\blambda^*\|+\norm{\bx^0-\bx^*}\big)+\sum_{i\in\mathcal{N}}\bigg[{1\over \tau_i}\|x^*_i-x_i^0\|^2+{4\over \kappa_i}\|\theta^*_i\|^2\bigg]$, and $\Theta_3(K)\triangleq 
8 N^2B^2\Gamma~\sum_{k=1}^K\alpha^{q_k}\left[2\gamma k^2+\left(\gamma+\tfrac{\norm{\blambda^*}}{\sqrt{N}B}\right)k\right]$.
Moreover, $\sup_{K\in\integers_+}\Theta_3(K)<\infty$; hence, $\tfrac{1}{K}\Theta_3(K)=\cO(\tfrac{1}{K})$.
\end{theorem}
\begin{remark}
Note that the suboptimality, infeasibility and consensus violation at the $K$-th iteration is $\cO(\Theta_3(K)/K)$, where $\Theta_3(K)$ denotes the error accumulation due to approximation errors, and $\Theta_3(K)$ can be bounded above for all $K\geq 1$ as $\Theta_3(K)\leq R \sum_{k=1}^K\alpha^{q_k} k^2$ for some constant $R>0$. Since $\sum_{k=1}^\infty\alpha^{\sqrt[\leftroot{-3}\uproot{3}p]{k}} k^2<\infty$ for any $p\geq 1$, if one chooses $q_k=\sqrt[\leftroot{-3}\uproot{3}p]{k}$ for $k\geq 1$, then the total number of communications per node until the end of $K$-th iteration can be bounded above by $\sum_{k=1}^K q_k=\cO(K^{1+1/p})$.
\end{remark}
In order to prove Theorem~\ref{thm:dynamic-error-bounds}, we first prove Theorem~\ref{thm:dynamic-rate} which help us to appropriately bound $\mathcal{L}(\bar{\bf x}^K,\by)-\mathcal{L}(\bx,\bar{\by}^K)$. Next we provide a technical result in Lemma~\ref{lemsum} to study the error accumulation.
\begin{theorem}\label{thm:dynamic-rate}
Let $\by=[\btheta^\top \bmu^\top]^\top$ such that $\bmu\in\reals^{n|\cN|}$, $\btheta=[\theta_i]_{i\in \mathcal{N}}\in\reals^m$, and $m\triangleq\sum_{i\in\cN}{m_i}$; and $\{\bx^k,\by^k\}_{k\geq 0}$ be the iterate sequence generated using Algorithm DPDA-D, displayed in Fig.~\ref{alg:PDD}, initialized from an arbitrary $\bx^0$ and $\by^0$; and $\{\be^k\}_{k\geq 1}$ be the proximal error sequence defined as in \eqref{eq:prox-error-seq}. Let $\bar{\bx}^{K}\triangleq\tfrac{1}{K}\sum_{k=1}^{K}\bx^k$, and $\bar{\by}^{K}\triangleq\tfrac{1}{K}\sum_{k=1}^{K}\by^k$ for $K\geq 1$. Suppose primal-dual step-sizes $\{\tau_i,\kappa_i\}_{i\in\cN}$ and $\gamma>0$ be chosen such that $\tau_i,\kappa_i>0$ and
{\small
\begin{align*}
\Big(\frac{1}{\tau_i}-L_i-\gamma\Big)\frac{1}{\kappa_i} \geq \sigma^2_{\max}(A_i),\quad \forall\ i\in\mathcal{N}.
\end{align*}
}%
Then, for any $\bx$ and $\by$, the following holds
{\small
\begin{align}\label{eq:dynamic-saddle-rate}
\mathcal{L}(\bar{\bf x}^K,\by)-\mathcal{L}(\bx,\bar{\by}^K) \leq &{1\over K} \big[D_x(\bx,\bx^0)+D_y(\by,\by^0)-\fprod{T(\bx-\bx^0),~\by-\by^0}\big] \nonumber \\
&+{1\over K}\sum_{k=0}^{K-1}\bigg[\|\be^{k+1}\| \big(2\gamma\sqrt{N}~B+\|\bmu-\bmu^{k+1}\|\big)\bigg],
\end{align}
}%
where $D_x$, $D_y$ are Bregman functions defined as in Definition~\ref{def:bregman}, $T=[A^\top A_0^\top]^\top$ for block-diagonal matrix $A\triangleq\diag([A_i]_{i\in\mathcal{N}})\in\reals^{m \times n|\cN|}$ and $A_0=\id_{n|\cN|}$.
\end{theorem}
\begin{proof}
For $k\geq 0$, let $\bq^k\triangleq 2{\bf x}^{k+1}-{\bf x}^k$. From strong convexity of $\sigma_C({\bmu})-\langle \bq^k,~{\bmu}\rangle+{1\over 2\gamma}\|{\bmu}-{\bmu}^k\|_2^2$ in $\bmu$ and the fact that ${\blambda}^{k+1}$ is its minimizer we conclude that
{\small
\begin{equation*}
\sigma_C(\bmu)-\langle \bq^k,~\bmu\rangle+\tfrac{1}{2\gamma}\|{\bmu}-{\bmu}^k\|^2 \geq \sigma_C({\blambda}^{k+1})-\langle \bq^k,~{\blambda}^{k+1}\rangle+\tfrac{1}{2\gamma}\|{\blambda}^{k+1}-{\bmu}^k\|^2+\tfrac{1}{2\gamma}\|\bmu-{\blambda}^{k+1}\|^2.
\end{equation*}}%
According to \eqref{eq:prox-error-seq}, ${\bmu}^{k}={\blambda}^{k}+\gamma \be^k$ for all $k\geq 1$; hence, from \eqref{eq:support-error} we have
{\small
\begin{eqnarray}
\lefteqn{\sigma_C({\bmu})-\langle \bq^k,~{\bmu}\rangle+\tfrac{1}{2\gamma}\|{\bmu}-{\bmu}^k\|_2^2 \geq}\nonumber\\
& & \sigma_C({\bmu}^{k+1})-\langle \bq^k,~{\bmu}^{k+1}\rangle +\tfrac{1}{2\gamma}\|{\bmu}^{k+1}-{\bmu}^k\|_2^2+\tfrac{1}{2\gamma}\|{\bmu}-{\bmu}^{k+1}\|_2^2- S^k,\label{eq:support-function-bound}
\end{eqnarray}
}%
where the error term $S^k$ is defined as
{\small
\begin{equation}
\label{eq:Sk}
S^k\triangleq\gamma \sqrt{N}~B \|\be^{k+1}\|-\gamma\|\be^{k+1}\|^2-\fprod{\be^{k+1},~\bmu-2\bmu^{k+1}+\bmu^k+\gamma\bq^k}.
\end{equation}
}%
If one customizes the steps of Lemma~\ref{lem:pda-lemma} for problem \eqref{eq:dynamic-saddle} using $\bmu^{k+1}$ instead of $\blambda^{k+1}$, it immediately follows from \eqref{eq:support-function-bound} that for all $k\geq 0$:
{\small
\begin{align}
\mathcal{L}&({\bf x}^{k+1},\by)-\mathcal{L}({\bf x},\by^{k+1})\leq S^k+ \left[D_x(\bx,\bx^k)+D_y(\by,\by^k)-\fprod{T(\bx-\bx^k),~\by-\by^k}\right] \label{eq:inexact-lagrangian-bound}\\
&-\left[D_x(\bx,\bx^{k+1})+D_y(\by,\by^{k+1})-\fprod{T(\bx-\bx^{k+1}),~\by-\by^{k+1}}\right]-\frac{1}{2}(\bz^{k+1}-\bz^k)^\top\mathbf{\bar{Q}}(\bz^{k+1}-\bz^k). \nonumber
\end{align}
}%
where $\bz^k=[{\bx^k}^\top {\by^k}^\top]^\top$, and $\mathbf{\bar{Q}}\triangleq\mathbf{\bar{Q}}(A,A_0)$ is defined as in Theorem~\ref{thm:general-rate} for $A_0=\id_{n|\cN|}$.

For $k\geq 0$, let $\bh^{k+1}\triangleq\mathcal{P}_C(\tfrac{1}{\gamma}\bmu^k+\bq^k)$; hence, $\blambda^{k+1}=\bmu^k+\gamma\bq^k-\gamma\bh^{k+1}$. Since $\bmu^{k+1}=\blambda^{k+1}+\gamma\be^{k+1}$, we have $\bmu^k+\gamma\bq^k-\bmu^{k+1}=\gamma(\bh^{k+1}-\be^{k+1})$; therefore, \eqref{eq:Sk} can be equivalently written as
\begin{equation}
\label{eq:Sk-bound}
S^k=\gamma\sqrt{N}~B\|\be^{k+1}\|-\fprod{\be^{k+1},~\bmu-\bmu^{k+1}+\gamma\bh^k}\leq\big(2\gamma\sqrt{N}~ B+\norm{\bmu-\bmu^{k+1}}\big)\norm{\be^{k+1}},
\end{equation}
where the inequality follows from Cauchy-Schwarz inequality and the fact that $\|\bh^{k+1}\|\leq \sqrt{N}~B$ since $\bh^{k+1}\in C$. Moreover, by setting $A_0=\id_{n|\cN|}$ instead of $M$ in the proof of Lemma~\ref{lem:schur}, one can easily show that $\mathbf{\bar{Q}}(A,A_0)\succeq 0$ when the condition in \eqref{eq:step-rule-dynamic} holds for all $i\in\cN$. Thus we can drop the last term in \eqref{eq:inexact-lagrangian-bound}. Therefore, similar to the proof of Theorem~\ref{thm:general-rate}, summing \eqref{eq:inexact-lagrangian-bound} over $k$ after bounding $S^k$ by \eqref{eq:Sk-bound}, dividing by $K$, and using Jensen's inequality gives the desired result. Indeed, since
$\mathbf{\bar{Q}}\succeq 0$, we can safely drop the last term, $D_x(\bx,\bx^{K})+D_y(\by,\by^{K})-\fprod{T(\bx-\bx^{K}),~\by-\by^{K}}\geq 0$, in the telescoping sum.
\end{proof}
The following lemma is slight extension of Proposition 3 in~\cite{chen2012fast}, where it is stated for $q=1$. The proof is omitted due to limited space.
\begin{lemma}\label{lemsum}
Let $\alpha\in(0,1)$, $q\geq 1$ is a rational number, and $d\in\integers_+$. Define $P(k,d)=\{\sum_{i=0}^d c_i k^i:\ c_i\in\reals\ i=1,\ldots,d \}$
denote the set of polynomials of $k$ with degree at most $d$. Suppose $p^{(k)}\in P(k,d)$ for $k\geq 1$, then
$\sum_{k=0}^{+\infty}p^{(k)}\alpha^{\sqrt[\leftroot{-3}\uproot{3}q]{k}}$ is finite.
\end{lemma}
Now we are ready to prove Theorem~\ref{thm:dynamic-error-bounds}.
\subsection{Proof of Theorem~\ref{thm:dynamic-error-bounds}}
Since \eqref{eq:approx_error} implies that $\|\mathcal{R}^k(\bx)-\mathcal{P}_C(\bx)\|\leq N~\Gamma \alpha^{q_k}\norm{\bx}$ for all $\bx$, and $\norm{\bx^k}\leq\sqrt{N}~B$ for $k\geq 1$, it follows from \eqref{eq:prox-error-seq} and \eqref{eq:mu-bound} that
{\small
\begin{align}
\label{eq:e-bound}
\|\be^{k+1}\|&=\|\cP_{\cC}\left(\tfrac{1}{\gamma}{\bmu}^k+2{\bf x}^{k+1}-{\bf x}^k\right)-\cR^k\left(\tfrac{1}{\gamma}{\bmu}^k+2{\bf x}^{k+1}-{\bf x}^k\right)\| \nonumber \\
&\leq N~\Gamma \alpha^{q_k}\norm{\tfrac{1}{\gamma}{\bmu}^k+2{\bf x}^{k+1}-{\bf x}^k}
\leq 4N^{\tfrac{3}{2}}B\Gamma \alpha^{q_k}(k+1).
\end{align}
}%

For $k\geq 1$, define $E^k(\bmu)\triangleq\|\be^{k}\| \big(2\gamma\sqrt{N}~B+\|\bmu-\bmu^{k}\|\big)$, which is the error term in \eqref{eq:dynamic-saddle-rate} due to approximating $\cP_{C}$ in the $k$-th iteration of the algorithm. The result of Theorem~\ref{thm:dynamic-rate} can be written more explicitly as follows: let $\bar{\bf x}^K\triangleq {1\over K}\sum_{k=1}^Kx_k$, $\bar{\bmu}^K\triangleq {1\over K}\sum_{k=1}^K\bmu^k$, and $\bar{\btheta}^K\triangleq {1\over K}\sum_{k=1}^K\btheta^k$, then for any ${\bf x},\bmu\in \mathbb{R}^{n|\mathcal{N}|}$, $\btheta\in \mathbb{R}^{m}$ for $m=\sum_{i\in\cN}m_i$, and for all $K\geq 1$, we have
{\small
\begin{align*} 
\mathcal{L}(\bar{\bf x}^K,\btheta,\bmu)-&\mathcal{L}({\bf x},\bar{\btheta}^K,\bar{\bmu}^K)\leq \Theta(\bx,\btheta,\bmu)/K,\\
\Theta(\bx,\btheta,\bmu)\triangleq&{1\over 2\gamma}\|\bmu-\bmu^0\|^2-\langle {\bf x}-{\bf x}^0,~\bmu-\bmu^0\rangle+\sum_{k=1}^{K}E^k(\bmu) \nonumber \\
&+\sum_{i\in\mathcal{N}}\bigg[{1\over 2\tau_i}\|x_i-x_i^0\|^2+{1\over 2\kappa_i}\|\theta_i-\theta_i^0\|^2-\langle A_i(x_i-x_i^0),\theta_i-\theta_i^0\rangle \bigg].\nonumber
\end{align*}
}%
Note that under the assumption in \eqref{eq:step-rule-dynamic}, Schur complement condition guarantees that
{\small
\begin{equation*}
\begin{bmatrix}
\frac{1}{\tau_i}\id_n &-A_i^\top\\
-A_i & \frac{1}{\kappa_i}\id_{m_i} \\
\end{bmatrix}
\preceq
\begin{bmatrix}
\frac{2}{\tau_i}\id_n & \mathbf{0}^\top\\
\mathbf{0} & \frac{2}{\kappa_i}\id_{m_i} \\
\end{bmatrix}.
\end{equation*}}%
Therefore,
{\small
\begin{eqnarray}
\label{eq:simple-C-bound-dynamic}
\Theta(\bx,\btheta,\bmu)\leq \sum_{i\in\mathcal{N}}\bigg[{1\over\tau_i}\|x_i-x_i^0\|^2+{1\over \kappa_i}\|\theta_i-\theta_i^0\|^2\bigg]+{1\over 2\gamma}\|\bmu-\bmu^0\|^2-\langle {\bf x}-{\bf x}^0,\bmu-\bmu^0\rangle+\sum_{k=1}^{K}E^k(\bmu).
\end{eqnarray}
}%
Let $({\bf x}^*,\btheta^*,\blambda^*)$ be an arbitrary saddle-point for $\cL$ in \eqref{eq:static-saddle}; hence, $\mathcal{L}(\bf x^*,\btheta^*,\blambda^*)=\Phi(\bx^*)$ and $\theta_i^*\in\cK_i^\circ$ for $i\in\cN$. As in the proof of Theorem~\ref{thm:static-error-bounds}, define $\tilde{\btheta}=[\tilde{\theta}_i]_{i\in\cN}$ such that $\tilde{\theta}_i\triangleq 2\|\theta_i^*\|
\big( \|\cP_{\mathcal{K}_i^\circ}(A_i\bar{x}_i^K-b_i)\|\big)^{-1}~\cP_{\mathcal{K}_i^\circ}(A_i\bar{x}_i^K-b_i)\in\cK_i^\circ$, which implies
{\small
\begin{equation}
\label{eq:tilde-theta-dynamic}
\langle A_i\bar{ x}_i^K -b_i, \tilde{\theta}_i \rangle=2\|\theta_i^*\| d_{\mathcal{K}_i}(A_i\bar{ x}_i^K -b_i).
\end{equation}}%

Define $\tilde{C}\triangleq\{{\bf x}\in \mathbb{R}^{n|\mathcal{N}|}: \exists \bar{x}\in\mathbb{R}^n \ \text{s.t.} \ x_i=\bar{x}, \forall i\in\mathcal{N} \}$. Note that $\tilde{C}$ is a closed convex cone, and the projection $\mathcal{P}_{\tilde{C}}(\bx)=\one\otimes \tilde{p}(\bx)$, where $\tilde{p}(\bx)$ is defined in \eqref{eq:proj-2}. Let $\tilde{\bmu}=2\norm{\blambda^*}{\mathcal{P}_{\tilde{C}^\circ}(\bar{\bf x}^k)\over \|\mathcal{P}_{\tilde{C}^\circ}(\bar{\bf x}^k)\|}\in\tilde{C}^\circ$, where $\tilde{C}^\circ$ denotes polar cone of $\tilde{C}$. Hence, it can be verified that $\langle \tilde{\bmu}, \bar{\bf x}^K\rangle=2\|\blambda^*\| d_{\tilde{C}}(\bar{\bf x}^K)$. 
Note that $\tilde{\bmu} \in \tilde{C}^\circ$ implies that $\sigma_{\tilde{C}}(\tilde{\bmu})=0$; moreover, we also have $C\subseteq \tilde{C}$; hence, $\sigma_C(\tilde{\bmu})\leq\sigma_{\tilde{C}}(\tilde{\bmu})=0$. Therefore, we can conclude that $\sigma_{C}(\tilde{\bmu})=0$ since $\zero\in C$. Together with \eqref{eq:tilde-theta-dynamic}, we get
{\small
\begin{equation}
\mathcal{L}(\bar{\bf x}^K,\tilde{\btheta},\tilde{\bmu})-\mathcal{L}(\bx^*,\btheta^*,\blambda^*)= \Phi(\bar{\bx}^K)-\Phi({\bx}^*)+2\left(\|\blambda^*\| d_{\tilde{C}}(\bar{\bf x}^K)+\sum_{i\in\cN} d_{\mathcal{K}_i}(A_i\bar{x}_i^K-b_i)\|\theta_i^*\|\right).
\end{equation}}%
Now we are going to upper bound $\Theta(\bx^*,\tilde{\btheta},\tilde{\bmu})$ using \eqref{eq:simple-C-bound-dynamic}. Since $\bmu^0=\mathbf{0}$, from Cauchy-Schwarz inequality,
{\small
\begin{equation}\label{eq:mu-x-product}
|\langle {\bf x}^*-{\bf x}^0,\tilde{\bmu}-\bmu^0\rangle| \leq 2\|\blambda^*\| \|{\bf x}^*-{\bf x}^0\|.
\end{equation}}%
Since $\btheta^*$ and $\blambda^*$ maximize the Lagrangian function at $\bx^*$, and $\btheta^0=\mathbf{0}$, it follows from \eqref{eq:tilde-theta-dynamic}, \eqref{eq:mu-x-product}, and \eqref{eq:simple-C-bound-dynamic} that
{\small
\begin{align*}
\label{eq:aux-upper-2}
\mathcal{L}(\bar{\bf x}^K,\tilde{\btheta},\tilde{\bmu})-\mathcal{L}(\bx^*,\btheta^*,\blambda^*)
&\leq \mathcal{L}(\bar{\bf x}^K,\tilde{\btheta},\tilde{\bmu})-\mathcal{L}(\bx^*,\bar{\btheta}^K,\bar{\bmu}^K)\\
&\leq \frac{1}{K}\Theta(\bx^*,\tilde{\btheta},\tilde{\bmu})\leq\frac{1}{K}\bigg(\Theta_2+\sum_{k=1}^{K}E^k(\tilde{\bmu})\bigg).
\end{align*}
}%
Therefore, we can conclude that
{\small
\begin{equation}
\label{eq:upper-bound-dynamic}
\Phi(\bar{\bx}^K)-\Phi({\bx}^*)+2\left(\|\blambda^*\| d_{\tilde{C}}(\bar{\bf x}^K)+\sum_{i\in\cN} d_{\mathcal{K}_i}(A_i\bar{x}_i^K-b_i)\|\theta_i^*\|\right)\leq \frac{1}{K}\bigg(\Theta_2+\sum_{k=1}^{K}E^k(\tilde{\bmu})\bigg),
\end{equation}}%
where we use $\mathcal{L}(\bx^*,\btheta^*,\blambda^*)=\Phi({\bx}^*)$ and the fact that $\sigma_{\mathcal{K}_i}(\tilde{\theta}_i)=0$ due to \eqref{eq:polar-support} since $\tilde{\theta}_i\in\cK_i^\circ$. Moreover, using \eqref{eq:e-bound} and \eqref{eq:mu-bound} we obtain that
{\small
\begin{align}
\sum_{k=1}^{K}E^k(\tilde{\bmu})&=\sum_{k=1}^K\|\be^{k}\| \big(2\gamma\sqrt{N}~B+\|\tilde{\bmu}-\bmu^{k}\|\big)
\leq\sum_{k=1}^K 4N^{\tfrac{3}{2}}B\Gamma \alpha^{q_k}k\big(2\gamma\sqrt{N}~B+\|\tilde{\bmu}-\bmu^{k}\|\big)\nonumber\\
&\leq 8N^2B^2\Gamma~\sum_{k=1}^K\alpha^{q_k}\left[2\gamma k^2
+\bigg(\gamma+\tfrac{\norm{\blambda^*}}{\sqrt{N}B}\bigg)k\right]=\Theta_3(K).\label{eq:E_mu_bound}
\end{align}}%
From Lemma~\ref{lemsum}, it follows that $\sup_{K\in\integers_+}\Theta_3(K)<\infty$.

Since $({\bf x}^*,\btheta^*,\blambda^*)$ is a saddle-point for $\cL$ in \eqref{eq:dynamic-saddle}, we clearly have $\mathcal{L}(\bar{\bf x}^K,\btheta^*,\blambda^*)-\mathcal{L}({\bf x}^*,\btheta^*,\blambda^*) \geq 0$; therefore,
{\small
\begin{equation}\label{eq:aux-lower-dynamic}
\Phi(\bar{\bf x}^K)-\Phi({\bf x}^*)+\langle \blambda^*,~\bar{\bf x}^K\rangle+\sum_{i\in\cN}\fprod{\theta_i^*,~A_i\bar{\bf x}_i^K-b_i}\geq 0.
\end{equation}
}%
As shown in the proof of Theorem~\ref{thm:static-error-bounds}, for all $i\in\cN$, we have
{\small
$$\langle A_i\bar{\bf x}_i^K-b_i, \theta_i^*\rangle\leq \|\theta^*_i\| d_{\mathcal{K}_i}(A_i\bar{\bf x}_i^K-b_i).$$}%
Similarly, we can also show that $\langle \blambda^*,~\bar{\bf x}^K\rangle\leq \|\blambda^*\| d_{\tilde{C}}(\bar{\bf x}^K)$.
Together with \eqref{eq:aux-lower-dynamic}, we conclude that
{\small
\begin{equation}\label{eq:lower-bound-dynamic}
\Phi(\bar{\bf x}^K)-\Phi({\bf x}^*)+\|\blambda^*\| d_{\tilde{C}}(\bar{\bf x}^K)+\sum_{i\in\cN}\|\theta^*_i\| d_{\mathcal{K}_i}(A_i\bar{\bf x}_i^K-b_i) \geq 0.
\end{equation}}%
By combining inequalities \eqref{eq:upper-bound-dynamic}, \eqref{eq:E_mu_bound} and \eqref{eq:lower-bound-dynamic} immediately implies the desired result.
\section{Dual consensus implementation}
\label{sec:dual}
In this section we study how to deal with resource allocation type problems of the following form:
{\small
\begin{equation}\label{eq:dual-implement}
\min_{\bxi} \bar{\Phi}(\bxi)\triangleq\sum_{i\in \cN}{\Phi_i(\xi_i)}\quad\hbox{s.t.}\quad\sum_{i\in\cN}R_i\xi_i-r_i \in \cK:\ y,
\end{equation}}%
where $\cK\subseteq \reals^m$ is a proper cone, $R_i\in\reals^{m\times n_i}$, and $r_i\in \reals^{m}$ are the problem data such that each node $i\in\cN$ only have access to $R_i$, $r_i$ and $\cK$ along with its objective $\Phi_i$ as defined in \eqref{eq:F_i}; and $\bxi=[\xi_i]_{i\in\cN}\in\reals^n$ is a long vector formed by local decisions $\xi_i\in\reals^{n_i}$ for each node $i\in\cN$, i.e., $n=\sum_{i\in\cN}n_i$, and $y\in\cK^\circ$ denotes the dual variable. One can reformulate $\eqref{eq:dual-implement}$ as the following saddle point problem:
{\small
\begin{equation}\label{eq:dual-saddle}
\min_{\substack{\bxi}}\max_{\substack{y}}\left\{\sum_{i\in \cN}{\Phi_i(\xi_i)}+\langle \sum_{i\in\cN}R_i\xi_i-r_i,~y \rangle-\sigma_{\cK}(y) \right\}.
\end{equation}}%
We assume that a dual optimal solution $y^*\in\cK^\circ$ exists and the duality gap is 0 for \eqref{eq:dual-implement}. Clearly these assumptions hold if \eqref{eq:dual-implement} satisfies Slater condition, i.e., there exists some $\bar{\bxi}\in\relint(\dom \bar{\Phi})$ such that $\sum_{i\in\cN}R_i\bar{\xi}_i-r_i \in \intr(\cK)$.
Suppose each node $i\in\cN$ has its own estimate $\theta_i\in\reals^m$ of a dual optimal solution. Let ${\by}=[y_i]_{i\in\cN}$; and given a bound $B_d>0$ such that $\norm{y^*}< B_d$, we define
{\small
\begin{equation}
\label{eq:Cd}
C_d\triangleq\{{\by}\in \mathbb{R}^{m|\mathcal{N}|}: \exists \bar{y}\in\mathbb{R}^m \ \mbox{ s.t. } \ y_i=\bar{y},\ \forall i\in\mathcal{N}, \ \|\bar{y}\|\leq B_d \}
\end{equation}}%
similarly as we defined $C$ in Section~\ref{sec:dynamic}. Since $\cG$ is {\it connected}, we can reformulate $\eqref{eq:dual-saddle}$ as a dual consensus formation problem:
{\small
\begin{equation}\label{eq:dual-saddle-dist}
\min_{\substack{\bxi}}\max_{\substack{\by\in C_d}}\ L(\bxi,\by)\triangleq\sum_{i\in \cN}\Big({\Phi_i(\xi_i)}+\langle R_i\xi_i-r_i,~y_i \rangle-\sigma_{\cK}(y_i) \Big),
\end{equation}}%
which follows from the definitions of $\sigma_{\cK}(\cdot)$ and $C_d$. Define $\cL:\reals^n\times\reals^{m|\cN|}\times\reals^{m|\cN|}\rightarrow\reals\cup\{+\infty\}$ such that
{\small
\begin{equation}
\label{eq:lagrangian-dual-implementation}
\cL(\bxi,\bw,\by)\triangleq\sum_{i\in \cN}\Big({\Phi_i(\xi_i)}+\langle R_i\xi_i-r_i,~y_i \rangle-\sigma_{\cK}(y_i)\Big)-\langle \bw,{\by} \rangle +\sigma_{C_d}(\bw).
\end{equation}}%
Note that for any $\bxi\in\dom\Phi$, we have $\max_{\by\in C_d}L(\bxi,\by)=\max_{\by}\min_{\bw}\cL(\bxi,\bw,\by)$; hence, $\eqref{eq:dual-saddle-dist}$ can be equivalently written as follows:
{\small
\begin{equation}\label{eq:dual-dist-problem}
\min_{\substack{\bxi}}\left\{\max_{\substack{\by}}\min_{\substack{\bw}}\cL(\bxi,\bw,\by)\right\}\
=\ \min_{\substack{\bxi,\bw}}\max_{\by}\cL(\bxi,\bw,\by),
\end{equation}}%
where the equality directly follows from Fenchel duality; indeed, interchanging inner $\max_{\btheta}$ and $\min_{\bw}$ is justified since $\intr(\cK)\neq\emptyset$ (see Theorem~3.3.5 in~\cite{borwein2010convex} -- second condition there holds because $C_d\cap \Pi_{i\in\cN}\intr(\cK)\neq\emptyset$). 

To solve \eqref{eq:dual-implement}, we can equivalently solve 
$\min_{\substack{\bxi,\bw}}\max_{\by}\cL(\bxi,\bw,\by)$, which is almost in the same form stated in Theorem~\ref{thm:general-rate}. Indeed, for $\bx=[\bxi^\top~\bw^\top]^\top$, 
let $\Phi(\bx)\triangleq\bar{\Phi}(\bxi)+\sigma_{C_d}(\bw)$, and $h(\by)\triangleq\sum_{i\in\cN}\sigma_{\mathcal{K}_i}(y_i)+\fprod{r_i,y_i}$; and define the block-diagonal matrix $R\triangleq\diag([R_i]_{i\in\mathcal{N}})\in\reals^{m|\cN|\times n}$ and
$T=
\begin{bmatrix}
R & -\id_{m|\cN|}
\end{bmatrix}
$. Moreover, given parameters $\gamma>0$, $\kappa_i,\tau_i>0$ for $i\in\cN$, similar to Definition~\ref{def:bregman}, let $\mathbf{D}_\tau=\diag([\frac{1}{\tau_i}\id_{n_i}]_{i\in\cN})$, $\mathbf{D}_\gamma=\frac{1}{\gamma}\id_{m|\cN|}$, and $\mathbf{D}_\kappa=\diag([\frac{1}{\kappa_i}\id_{m_i}]_{i\in\cN})$.
Defining $\psi_x(\bx)=\frac{1}{2}\bxi^\top\mathbf{D}_\tau\bxi+\frac{1}{2}\bw^\top\mathbf{D}_\gamma\bw$ and $\psi_y(\by)=\frac{1}{2}\by^\top\mathbf{D}_\kappa\by$ leads to the following Bregman distance functions: $D_x(\bx,\bar{\bx})=\frac{1}{2}\norm{\bxi-\bar{\bxi}}_{\mathbf{D}_\tau}^2+\frac{1}{2}\norm{\bw-\bar{\bw}}_{\mathbf{D}_\gamma}^2$, and $D_y(\by,\bar{\by})=\frac{1}{2}\norm{\by-\bar{\by}}_{\mathbf{D}_\kappa}^2$.

Therefore, given the initial iterates $\bx^0,\bw^0,\by^0$ and parameters $\gamma>0$, $\tau_i,\kappa_i>0$ for $i\in\cN$, choosing Bregman functions $D_x$ and $D_y$ as defined above, and setting $\nu_x=\nu_y=1$, PDA iterations given in \eqref{eq:PDA-x}-\eqref{eq:PDA-y} can be written explicitly as follows: Setting $\bv^0\gets\bw^0$, for $i\in\cN$ compute
{\small
\begin{align}\label{eq:pock-pd-3}
&\xi_i^{k+1}\gets\argmin_{\xi_i} \rho_i(\xi_i) +f_i(\xi_i^k)+\langle\nabla f_i(\xi_i^k),\xi_i-\xi_i^k \rangle + \langle R_i\xi_i-r_i,y_i^k\rangle +{1\over 2\tau_i}\|\xi_i-\xi_i^k\|^2_2,\nonumber \\
&\bw^{k+1}\gets\argmin_{\bw} \sigma_{C_d}(\bw)-\langle {\by}^k,~\bw \rangle +{1\over 2\gamma}\|\bw-\bv^k\|_2^2,\qquad \bv^{k+1}\gets\bw^{k+1},\\
&y_i^{k+1}\gets\argmin_{y_i}\sigma_{\mathcal{K}}(y_i) -\langle R_i(2\xi_i^{k+1}-\xi_i^k)-r_i,~y_i\rangle +\langle 2v_i^{k+1}-v_i^k,~y_i\rangle+{1\over 2\kappa_i}\|y_i-y_i^k\|_2^2.\nonumber
\end{align}
}
Using arguments similar to those in Section~\ref{sec:dynamic}, we can equivalently write $\bw$-update as follows:
{\small
\begin{equation}
\label{eq:exact-w-computation}
\bw^{k+1}\gets\bv^k+\gamma\by^k-\gamma~\mathcal{P}_{C_d}\left( \frac{1}{\gamma}\bv^k+\by^k\right).
\end{equation}}%
As in Section~\ref{sec:dynamic}, define $\cB_0\triangleq\{y\in\reals^m:\ \norm{y}\leq B_d\}$, and let $\mathcal{B}\triangleq\{\by: \ \|y_i\|\leq B_d,\ i\in\cN\}=\Pi_{i\in\cN}\cB_0$. Consider the $k$-th iteration of PDA as shown in \eqref{eq:pock-pd-3}. Instead of computing $\bw^{k+1}$ exactly according to \eqref{eq:exact-w-computation}, we approximate $\bw^{k+1}$ with the help of Lemma~\ref{lem:approximation} and set $\bv^{k+1}$ to this approximation. In particular, let $t_k$ be the total number of consensus steps done before iteration $k$ of PDA shown in \eqref{eq:pock-pd-3}, and let $q_k\geq 1$ be the number of consensus steps within iteration $k$. For $\bw=[w_i]_{i\in\cN}$, define
{\small
\begin{equation}
\label{eq:approx-average-dual}
\cR^k(\bw)\triangleq\mathcal{P}_{\mathcal{B}}\left((W^{t_k+q_k,t_k}\otimes\id_m)~\bw\right)
\end{equation}
}%
to approximate $\mathcal{P}_{C_d}(\bx)$ in \eqref{eq:exact-w-computation}. Thus, $\mathcal{R}^k(.)$ can be computed in a distributed fashion requiring $q_k$ communications with the neighbors for each node. 
In particular, components of $\cR^k(\bw)$ can be computed at each node as follows:
{\small
\begin{equation}
\label{eq:component-approx-dual}
\cR^k(\bw)=[\cR_i^k(\bw)]_{i\in\cN}\quad \hbox{such that}\quad \cR_i^k(\bw)\triangleq\cP_{\cB_0}(\sum_{j\in\cN}W^{t_k+q_k,t_k}_{ij}w_j).
\end{equation}
}%
Now, using \eqref{eq:approx-average-dual}, we approximate $\bw^{k+1}$ computation in \eqref{eq:exact-w-computation} with the following update rule:
{\small
\begin{equation}\label{eq:inexact-rule-dual}
{\bv}^{k+1}\gets{\bv}^k+\gamma\by^k-\gamma\cR^k\left(\tfrac{1}{\gamma}{\bv}^k+\by^k\right),
\end{equation}
}%
and replace the exact computation $\bv^{k+1}\gets\bw^{k+1}$ in \eqref{eq:pock-pd-3} with the inexact iteration rule in \eqref{eq:inexact-rule-dual}. Thus, PDA iterations given in~\eqref{eq:pock-pd-3}, for solving distributed resource allocation problem in \eqref{eq:dual-dist-problem}, can be computed inexactly, but in \emph{decentralized} way for dynamic connectivity, via the node-specific computations as in Algorithm~DPDA-R displayed in Fig.~\ref{alg:PDDual} below.
\begin{figure}[htpb]
\centering
\framebox{\parbox{\columnwidth}{
{\small
\textbf{Algorithm DPDA-R} ( $\bxi^{0},\by^0,\gamma,\{\tau_i,\kappa_i\}_{i\in\cN}$ ) \\[1.5mm]
Initialization: $v_i^0\gets\mathbf{0}$, \quad $i\in\cN$\\
Step $k$: ($k \geq 0$)\\
\text{ } 1. $\xi_i^{k+1}\gets\prox{\tau_i\rho_i}\bigg(\xi_i^k-\tau_i\bigg(\nabla f_i(\xi_i^k)+R_i^\top y^k_i\bigg)\bigg), \quad i\in\cN$  \\
\text{ } 2. $v_i^{k+1}\gets v_i^k+\gamma y_i^k-\gamma\cR_i^k\left(\tfrac{1}{\gamma}\bv^k+\by^k\right), \quad i \in \cN$\\
\text{ } 3. $y_i^{k+1}\gets\cP_{\cK_i^\circ}\bigg(\theta_i^k+\kappa_i\Big(R_i(2x_i^{k+1}-x_i^k)-(2v_i^{k+1}+v_i^k)-b_i\Big)\bigg), \quad i \in \cN$\\
 }}}
\caption{\small Distributed Primal Dual Algorithm for dual implemtation}
\label{alg:PDDual}

\end{figure}

\begin{theorem}\label{col:dual-error-bound}
Let $({\bf \bxi}^*,\bw^*,\by^*)$ be an arbitrary saddle-point for $L$ in \eqref{eq:lagrangian-dual-implementation}. Starting from $\bv^0=\mathbf{0}$, $\by^0=\mathbf{0}$, and an arbitrary $\bxi^0$, let $\{\bxi^k,\bv^k,\by^k\}_{k\geq 0}$ be the iterate sequence generated using Algorithm DPDA-R, displayed in Fig.~\ref{alg:PDDual}, using $q_k=\sqrt[\leftroot{-3}\uproot{3}p]{k}$ consensus steps at the $k$-th iteration for all $k\geq 1$ for some rational $p\geq 1$. Let primal-dual step-sizes $\{\tau_i,\kappa_i\}_{i\in\cN}$ and $\gamma$ be chosen such that $\tau_i>\frac{1}{L_i}$, $\kappa_i>\frac{1}{\gamma}$, and \vspace*{-3mm}
{\small
\begin{equation}\label{eq:step-rule-dual}
\left({1\over \tau_i}-L_i\right)\left( {1\over \kappa_i}-\gamma\right) \geq \sigma_{\max}^2(R_i),\quad \forall\ i\in\mathcal{N}.
\end{equation}
}%
Then, the sequence $\{\bar{\bxi}^K\}$, where $\bar{\bxi}^K\triangleq {1\over K}\sum_{k=1}^K\bxi^k$, has a limit point $\bxi^*$ which is an optimal solution to \eqref{eq:dual-implement}. In particular, the following holds
{\small
\begin{equation*}
\norm{\by^*}d_{\mathcal{K}}\left( \sum_{i\in\cN}R_i\bar{\bxi}_i^K-r_i\right)\leq \frac{\Theta_4+\Theta_5(K)}{K},\qquad |\Phi(\bar{\bxi}^k)-\Phi({\bxi}^*)|\leq \frac{\Theta_4+\Theta_5(K)}{K},
\end{equation*}
}%
where 
$\Theta_4\triangleq\|\bw^*\|\big({1\over 2\gamma}\|\bw^*\|+2\norm{\by^*}\big)+\sum_{i\in\mathcal{N}}\bigg[{1\over \tau_i}\|\xi^*_i-\xi_i^0\|^2+{4\over \kappa_i}\|y^*_i\|^2\bigg]$, and 
$\Theta_5(K)\triangleq 2 N^2B^2_d\Gamma \sum_{k=1}^K \alpha^{q_k}k \left(2\gamma (k+1)+{\norm{\bw^*}\over \sqrt{N}B_d}\right)$.
Moreover, $\sup_{K\in\integers_+}\Theta_5(K)<\infty$; hence, $\tfrac{1}{K}\Theta_5(K)=\cO(\tfrac{1}{K})$.
\end{theorem}
\begin{remark}
The proof of this result can be given using similar arguments as in Section~\ref{sec:dynamic} after a few modifications. Indeed, recall that to prove Theorem~\ref{thm:dynamic-error-bounds}, we first prove Theorem~\ref{thm:dynamic-rate}. Similarly, we need to prove an analogous result to Theorem~\ref{thm:dynamic-rate} which help us to appropriately bound $\mathcal{L}(\bar{\bxi}^K,\bar{\bv}^K,\by)-\mathcal{L}(\bx,\bv,\bar{\by}^K)$. This analogous result requires
{\small $\mathbf{\bar{Q}}\triangleq
\begin{bmatrix}
\mathbf{\bar{D}}_\tau & 0& -R^\top\\
0 & \mathbf{D}_\gamma & \bI\\
-R & \bI & \mathbf{D}_\kappa
\end{bmatrix}\succeq 0$}, where $\mathbf{\bar{D}}_\tau\triangleq\mathbf{D}_\tau-\diag([L_i\id_{n_i}]_{i\in\cN})=\diag([({1\over \tau_i} -L_i)\id_{n_i}]_{i\in \mathcal{N}})$.
\end{remark}
In the next lemma, we show that if $\{\tau_i,\kappa_i\}_{i\in\cN}$ and $\gamma$ are chosen satisfying the condition \eqref{eq:step-rule-dual} in Theorem~\ref{col:dual-error-bound}, then $\mathbf{\bar{Q}}\succeq 0$ holds.
\begin{lemma}\label{lem:schur-dual}
Given $\{\tau_i,\kappa_i\}_{i\in\mathcal{N}}$ and $\gamma$ such that $\gamma>0$, and $\tau_i>0$, $\kappa_i>0$ for $i\in\cN$. Then $\mathbf{\bar{Q}}\succeq 0$ if $\{\tau_i,\kappa_i\}_{i\in\mathcal{N}}$ and $\gamma$ are chosen such that $\tau_i>\frac{1}{L_i}$,  $\kappa_i>\frac{1}{\gamma}$, and \eqref{eq:step-rule-dual} holds. \vspace*{-2mm}
\end{lemma}
\begin{proof}
Since $\mathbf{D}_\kappa\succ 0$, Schur complement condition implies that $\bar{\bQ}\succeq 0$ if and only if
{\small
\begin{equation*}
B-\begin{bmatrix}-R^\top \\ \bI_{n|\cN|} \end{bmatrix}\bD_\kappa^{-1} \begin{bmatrix} -R & \bI_{n|\cN|}\end{bmatrix}\succeq 0,\quad \hbox{where}\quad B=\begin{bmatrix} \bar{\bD}_\tau & 0 \\ 0 & \bD_\gamma\end{bmatrix};
\end{equation*}}%
hence, the above condition can be equivalently written as
{\small
\begin{equation}\label{matrix-dual}
\begin{bmatrix} \bar{\bD}_\tau-R^\top\bD^{-1}_\kappa R & R^\top \bD_\kappa^{-1} \\ \bD^{-1}_\kappa R & \bD_\gamma-\bD_\kappa^{-1} & \end{bmatrix}\succeq 0.
\end{equation}}%
Again using Schur complement one more time, one can conclude that \eqref{matrix-dual} holds if $\bar{\bD}_\tau-R^\top\bD^{-1}_\kappa R-R^\top \bD^{-1}_\kappa (\bD_\gamma -\bD_\kappa^{-1})^{-1}\bD^{-1}_\kappa R\succeq 0$ and $\bD_\tau-\bD_\kappa^{-1}\succ 0$. Since the matrices are block diagonal, we obtain the desired result immediately.
\end{proof}

Recall that the definition of $C_d$ in \eqref{eq:Cd} involves a bound $B_d$ such that $\norm{y^*}< B_d$ for some dual optimal solution $y^*$. Next, we show that given a Slater point we can find a ball containing the optimal dual set for problem \eqref{eq:dual-implement}. To this end, we will prove this result for a more general case where the feasible set can be described by a general convex function $g(\bxi)\triangleq\sum_{i\in\cN}g_i(\xi_i)\in {\cK}$, where $g_i(\xi_i)$ is private convex function of $i\in\cN$, and in particular $g_i(\xi_i)=R_i\xi_i-r_i$ in \eqref{eq:dual-implement}.

Let $\Phi:\reals^n\rightarrow\reals^n\cup\{+\infty\}$ and $g:\reals^n\rightarrow\reals^m$ be arbitrary functions of $\bx$, and $\cK\subset\reals^m$ is a cone. We do not assume convexity for $\Phi$, $g$, and $\cK$, which are the components of the following generic problem
{\small
\begin{equation}
\label{eq:generic problem}
\Phi^*\triangleq\min_\bx \Phi(\bxi)\quad \hbox{s.t.}\quad g(\bxi)\in\cK\ :\ y\in\cK^\circ,
\end{equation}}%
where $y\in\reals^m$ denotes the vector of dual variables. Let $q$ denote the dual function, i.e.,
{\small
\begin{equation}
q(y)=\left\{
       \begin{array}{ll}
         \inf_{\bxi}\Phi(\bxi)+y^\top g(\bxi), & \hbox{if $y\in\cK^\circ$;} \\
         -\infty, & \hbox{o.w.}
       \end{array}
     \right.
\end{equation}}%
We assume that there exists $\hat{y}\in\cK^\circ$ such that $q(\hat{y})>-\infty$. Since $q$ is a closed concave function, this assumption implies that $-q$ is a proper closed convex function. Next we show that for any $\bar{y}\in\dom q=\{y\in\reals^m:\ q(y)>-\infty\}$, the superlevel set $Q_{\bar{y}}\triangleq\{y \in \dom q:\ q(y)\geq q(\bar{y})\}$ is contained in a Euclidean ball centered at the origin, of which radius can be computed efficiently. A special case of this dual boundedness result is well known when $\cK=\reals^m_{+}$~\cite{uzawa58}, and its proof is very simple and based on exploiting the componentwise separable structure of $\cK=\reals^m_{+}$ -- see Lemma~1.1 in~\cite{Nedic08_1J}; however, it is not trivial to extend this result to our setting where $\cK$ is an \emph{arbitrary} cone with $\intr(\cK)\neq\emptyset$.
\begin{lemma}\label{dual-bound}
Let $\bar{\bxi}$ be a Slater point for \eqref{eq:generic problem}, i.e., $\bar{\bxi}\in\relint(\dom \Phi)$ such that $g(\bar{\bxi})\in\intr(\cK)$. Then for all $\bar{y}\in \dom q$, the superlevel set $Q_{\bar{y}}$ is bounded. In particular,
{\small
\begin{equation}
\label{eq:dual-bound-radius}
\norm{y}\leq {\Phi(\bar{\bxi})-q(\bar{y}) \over r^*},\quad \forall y \in Q_{\bar{y}},
\end{equation}}%
where $0<r^*=\min_w\{w^\top g(\bar{\bxi}):\ \|w\|= 1,\ w\in \cK^*\}$.
Note that this is not a convex problem; instead, one can upper bound \eqref{eq:dual-bound-radius} using $0<\tilde{r}\leq r^*$, which can be efficiently computed by solving a convex problem
{\small
\begin{equation}\label{dual-bound-problem-tight}
\tilde{r}\triangleq\min_w\{w^\top g(\bar{\bxi}):\ \|w\|_1= 1,\ w\in \cK^*\}.
\end{equation}}%
\end{lemma}
\begin{proof}
For any $y\in Q_{\bar{y}}\subset\cK^\circ$ we have that
{\small
\begin{equation}
q(\bar{y})\leq q(y)=\inf_{\bxi} \{\Phi(\bxi)+y^\top g(\bxi)\} \leq \Phi(\bar{\bxi})+y^\top g(\bar{\bxi}) \label{eq3},
\end{equation}}%
which implies that $-y^\top g(\bar{\bxi})\leq \Phi(\bar{\bxi})-q(\bar{y})$.
Since $g(\bar{\bxi})\in\intr(\cK)$ and $y\in\cK^\circ$, we clearly have $-y^\top g(\bar{\bxi})>0$ whenever $y\neq\zero$. Indeed, since $g(\bar{\bxi})\in\intr(\cK)$, there exist a radius $r>0$ such that $g(\bar{\bxi})+ru\in \cK$ for all $\|u\|\leq 1$. Hence, for $y\neq\zero$, by choosing $u=y/\|y\|$ and using the fact that $y\in\cK^\circ$, we get that $0\geq(g(\bar{\bxi})+r y/\|y\|)^\top y$. Therefore, \eqref{eq3} implies that for all $y\in Q_{\bar{y}}$ we have
{\small
\begin{equation}
r\|y\| \leq -y^\top g(\bar{\bxi})\leq \Phi(\bar{\bxi})-q(\bar{y})\quad \Rightarrow\quad \|y\| \leq {\Phi(\bar{\bxi})-q(\bar{y}) \over r}.
\end{equation}}%
Now, we will characterize the largest radius $r^*>0$ such that $\cB(g(\bar{\bxi}),r^*)\triangleq\{g(\bar{\bxi})+r^*u:\ \norm{u}\leq 1\}\subset\cK$. Note that $r^*>0$ can be written explicitly as the optimal value of the following optimization problem:
{\small
\begin{equation}
r^*=\max\{r:\ d_\cK\big(g(\bar{\bx})+ru\big)\leq 0,\quad \forall u:~ \|u\| \leq 1\} \label{radius-problem}.
\end{equation}}%
Let $\gamma(r)\triangleq\sup\{d_\cK\big(g(\bar{\bxi})+ru\big):\ \norm{u}\leq 1\}$; hence, $r^*=\max\{r:\ \gamma(r)\leq 0\}$. Note that for any fixed $u\in\reals^m$, $d_\cK\big(g(\bar{\bxi})+ru\big)$ as a function of $r$ is a composition of convex function $d_\cK(.)$ with affine function in $r$; hence, it is convex in $r\in\reals$ for all $u\in\reals^m$. Moreover, since supremum of convex functions is also convex, $\gamma(r)$ is convex in $r$. From the definition of $d_{\cK}(\cdot)$, we have
{\small
\begin{equation}
\gamma(r)=\sup_{\|u\|\leq 1}\inf_{\bxi\in \cK}\|\bxi-\big(g(\bar{\bxi})+ru\big)\|
=\sup_{\|u\|\leq 1}\inf_{\bxi\in\cK} \sup_{\|w\|\leq 1}w^\top \Big(\bxi-\big(g(\bar{\bxi})+ru\big)\Big) \label{eq5}.
\end{equation}}%
Since $\{w\in\reals^m:\ \|w\|\leq 1\}$ is a compact set, and the function in \eqref{eq5} is a bilinear function of $w$ and $\bxi$ for each fixed $u$, we can interchange $\inf$ and $\sup$, and obtain
{\small
\begin{align}
\gamma(r)=&\sup_{\|u\|\leq 1}\sup_{\|w\|\leq 1} \inf_{\bxi\in \cK} w^\top \Big(\bxi-\big(g(\bar{\bxi})+ru\big)\Big) \nonumber \\
= &\sup\limits_{\substack{\|u\|\leq 1 \\ \|w\| \leq 1\\ w\in \cK^*}} -w^\top \big(g(\bar{\bx})+ru\big) = \sup\limits_{\substack{\|w\| \leq 1\\ w\in \cK^*}} -w^\top g(\bar{\bxi})+r\|w\|. \label{gamma-radius-problem}
\end{align}}%
Let $w^*(r)$ be an $\argmax$ of \eqref{gamma-radius-problem}. It is easy to see that $\norm{w^*(r)}=1$, since the supremum of a convex function over a convex set is attained on the boundary of the set. Therefore,
{\small
\begin{equation}
\gamma(r)=\sup\limits_{\substack{\|w\|= 1\\ w\in \cK^*}} -w^\top g(\bar{\bxi})+r. \label{eq:gamma-explicit}
\end{equation}}%
Since $r^*=\max\{r:\ \gamma(r)\leq 0\}$, from \eqref{eq:gamma-explicit} it follows that
{\small
\begin{equation*}
(P_1)\qquad r^*=\max\Big\{r:\ r\leq -\sup\{ -w^\top g(\bar{\bxi}):\ \|w\|= 1,\ w\in \cK^*\}\Big\}=
\min\limits_{\substack{\|w\|= 1\\ w\in \cK^*}} w^\top g(\bar{\bxi}).
\end{equation*}}%
Note that $(P_1)$ is not a convex problem due to boundary constraint, $\|w\|= 1$. Next, we define a convex problem $(P_2)$ to lowerbound $r^*$ so that we can upper bound the right hand side of \eqref{eq:dual-bound-radius}.
{\small
\begin{equation*}
(P_2)\qquad \min\limits_{\substack{\|w\|_1= 1\\ w\in \cK^*}} w^\top g(\bar{\bxi}) \leq r^*= \min\limits_{\substack{\|w\|= 1\\ w\in \cK^*}} w^\top g(\bar{\bxi}). %
\end{equation*}}%
Let $w^*$ be an optimal solution to $(P_1)$ and define $\bar{w}=w^*/\|w^*\|_1$. Clearly, $\|\bar{w}\|_1=1$ and $\bar{w}\in \cK^*$. Moreover, since $\|w^*\|_1\geq \|w^*\|=1$ we have that
{\small
\begin{equation*}
\tilde{r}=\min\limits_{\substack{\|w\|_1= 1\\ w\in \cK^*}} w^\top g(\bar{\bxi}) \leq\bar{w}^\top g(\bar{\bxi})=\frac{1}{\|w^*\|_1} {w^{*}}^\top g(\bar{\bx}) \leq {w^{*}}^\top g(\bar{\bxi})=\min\limits_{\substack{\|w\|= 1\\ w\in \cK^*}} w^\top g(\bar{\bxi}).
\end{equation*}}%
\end{proof}
\begin{remark}
Consider the problem in \eqref{eq:generic problem}. If we further assume that $\Phi$ is convex, $-g$ is $\cK$-convex, and $\cK$ is a proper cone, and $\Phi^*>-\infty$, i.e., \eqref{eq:generic problem} is a convex problem with a finite optimal value, then it is known that the Slater condition in Lemma~\ref{dual-bound} is sufficient for the existence of a dual optimal solution, $y^*\in\cK^\circ$, and for zero duality gap. Hence, the dual optimal solution set $Q^*\triangleq\{y\in\cK^\circ:\ q(y)\geq \Phi^*\}$ can be bounded using Lemma~\ref{dual-bound}. In particular, $\norm{y}\leq\big(\Phi(\bar{\bxi})-\Phi^*\big)/r^*$ for all $y\in Q^*$.
\end{remark}
\begin{remark}
Let $g_j:\reals^n\rightarrow\reals$ be the components of $g:\reals^n\rightarrow\reals^m$ for $j=1,\ldots,m$, i.e., $g(\bxi)=[g_j(\bxi)]_{j=1}^m$. When $\cK=\reals^m_+$, Lemma~1.1 in~\cite{Nedic08_1J} implies that for any $\bar{y}\in\dom q$ and $\bar{\bxi}$ such that $g(\bar{\bxi})\in\intr(\cK)$, i.e., $g_j(\bar{\bxi})>0$ for all $j=1,\ldots,m$, the superlevel set $Q_{\bar{y}}$ can be bounded as follows $\norm{y}\leq \big(\Phi(\bar{\bxi})-q(\bar{y})\big)/\bar{r}$ for all $y \in Q_{\bar{y}}$, where $\bar{r}\triangleq\min\{g_j(\bar{\bxi}):\ j=1,\ldots,m\}$. Note our result in Lemma~\ref{dual-bound} gives the same bound since $r^*=\min_w\{w^\top g(\bar{\bxi}):\ \|w\|= 1,\ w\in\reals^m_+\}=\bar{r}$.
\end{remark}
\section{Numerical Section}
\label{sec:numerics}
We test DPDA-S and DPDA-D on a primal linear SVM problem where the data is distributed among computing nodes.
Consider a random connected graph $G=(\mathcal{N},\mathcal{E})$ and $N\triangleq |\mathcal{N}|$. Let $\cS\triangleq\{1,2,..,s\}$ and $\cD\triangleq\{(x_\ell,y_\ell)\in\reals^n\times\{-1,+1\}:\ \ell\in\cS\}$ be a set of feature vector and label pairs. Suppose $\cS$ is partitioned into $\cS_\mathrm{test}$ and $\cS_\mathrm{train}$, i.e., the index sets for the test and training data; let $\{\cS_i\}_{i\in\cN}$ be a partition of $\cS_\mathrm{train}$ among the nodes $\cN$. Let $\bw=[w_i]_{i\in\cN}$, $\bb=[b_i]_{i\in\cN}$, and $\xi\in\reals^{|\cS_\mathrm{train}|}$ such that $w_i\in\reals^n$ and $b_i\in\reals$ for $i\in\cN$.

Consider the following distributed SVM problem:
{\small
\begin{align}
\min_{\bw,\bb,\xi}\  &{1\over 2}\sum_{i\in \mathcal{N}}\|w_i\|^2+NC\sum_{i\in\mathcal{N}}\sum_{\ell\in\cS_i}\xi_\ell \nonumber \\
\hbox{s.t.} \ \ &y_\ell(w_i^Tx_\ell+b_i) \geq 1-\xi_\ell,\quad \xi_\ell\geq 0,\quad \ell\in\cS_i,\ i\in\mathcal{N},\nonumber \\
&w_i=w_j,\quad b_i=b_j \quad (i,j)\in\cE. \nonumber
\end{align}
}%
Similar to \cite{forero2010consensus}, $\{x_\ell\}_{\ell\in\cS}$ is generated from two-dimensional multivariate Gaussian distribution with covariance matrix $\Sigma=[1,0;0,2]$ and with mean vector either $m_1=[-1,-1]^T$ or $m_2=[1,1]^T$ with equal probability. The experiment was performed for $s=900$ such that $|\cS_\mathrm{train}|=300$ and $|\cS_\mathrm{train}|=600$. We examine both DPDA-S and DPDA-D in four cases depending on $(i)$ parameter $C\in\{2,10\}$; and $(ii)$ algebraic connectivity of the network graph to be 0.05 and 1. For each of these situations the five replications was performed with the same data set, statistics from each replication and their average over replications are plotted. In particular, for each case, the corresponding suboptimality, feasibility and consensus violation is plotted against iteration counter, where consensus violation is defined as $\max_{(i,j)\in\mathcal{E}}\|[w_i^\top b_i]^\top-[w_j^\top b_j]^\top\|$. Fig.~\ref{Static-Alg0.05}, corresponds to increasing $C$ from $C=2$ to $C=10$ for static network topology with $N=10$, and the algebraic connectivity is 0.05. 
The other figures corresponding to the rest of the test cases are given in Fig~\ref{Static-Alg0.05}, Fig.~\ref{Dynamic-Alg0.05}, Fig~\ref{Static-Alg1}, and Fig.~\ref{Dynamic-Alg1}. Furthermore, visual comparison between DPDA-S, local SVM and centralized SVM for the same data set with $C=10$ is given in Fig.~\ref{SVM}.

\section{Conclusion}
We propose primal-dual algorithms for distributed optimization subject to agent specific conic constraints and/or global conic constraints with separable local components. By assuming composite convex structure on the primal functions, we show that our proposed algorithms converge with $\mathcal{O}(1/{k})$ rate where $k$ is the number of consensus iterations. To the best of our knowledge, this is the best rate result for our setting. We would like to emphasize that using these techniques, we can solve more general distributed optimization problems when there both local and global decisions subject to both local constraints and global resource sharing constraints, i.e.,
{\small
\begin{equation}\label{eq:extension-problem}
\min_{\substack{x,\bxi}}\left\{\sum_{i\in \cN}{\Phi_i(x,\xi_i)}:\ \sum_{i\in\cN}Q_ix+ R_i\xi_i-r_i \in \cK,\  \sum_{i\in\cN}A_ix+B_i\xi_i-b_i \in \cK_i,\quad i\in\cN, \right\}
\end{equation}}%
where $x\in\reals^n$ denotes the global variable and $\xi_i$ denotes the local variable, $\cK\subseteq \reals^m$ and $\cK_i\in\reals^{m_i}$ are proper cones, $Q_i\in\reals^{m\times n}$, $R_i\in\reals^{m\times n_i}$, $r_i\in \reals^{m}$, $A_i\in\reals^{m_i\times n}$, $B_i\in\reals^{m_i\times n}$, $b_i\in\reals^{m_i}$ are the problem data such that each node $i\in\cN$ only have access to $Q_i$, $R_i$, $r_i$, $A_i$, $B_i$, $b_i$, $\cK_i$ and $\cK$ along with its objective $\Phi_i$ as defined in \eqref{eq:F_i}.

\begin{figure}[h]
\begin{subfigure}[b]{0.33\textwidth}
\includegraphics[scale=0.23]{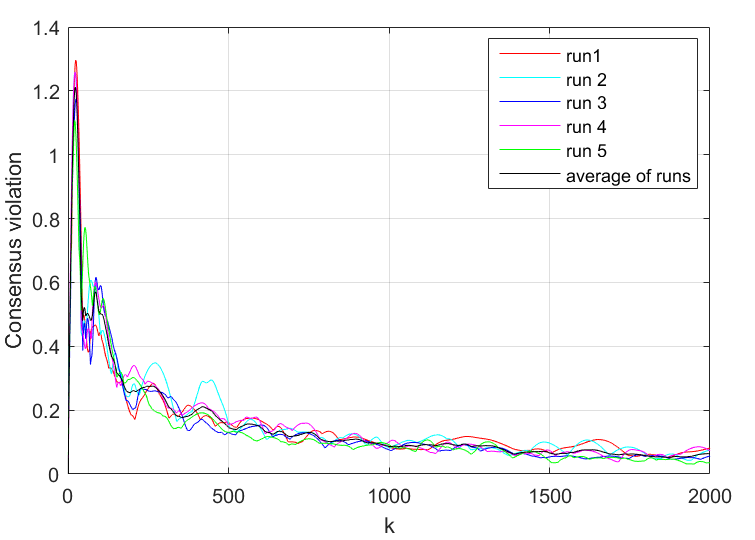}
\end{subfigure}
\begin{subfigure}[b]{0.33\textwidth}
\includegraphics[scale=0.235]{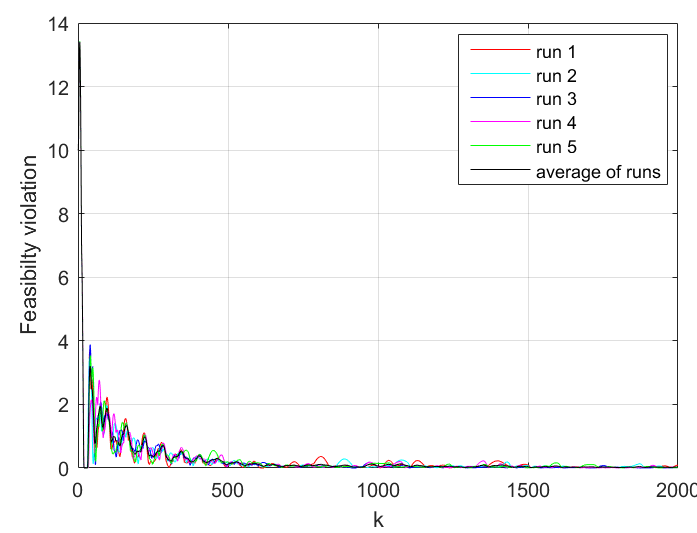}
\end{subfigure}
\begin{subfigure}[b]{0.33\textwidth}
\includegraphics[scale=0.235]{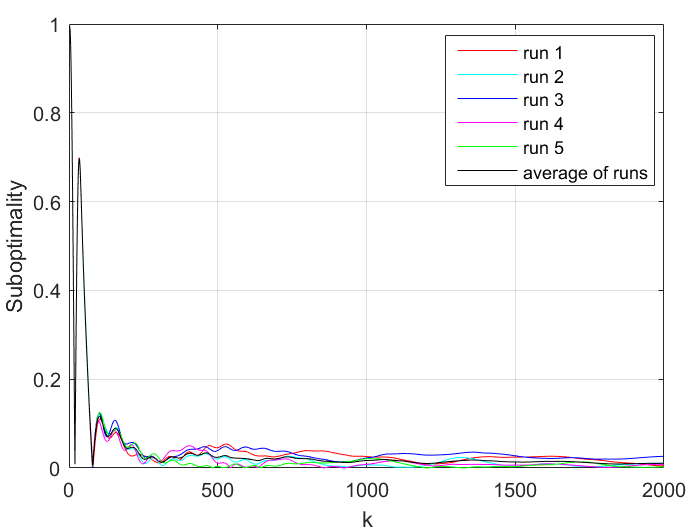}
\end{subfigure}
\\
\begin{subfigure}[b]{0.33\textwidth}
\includegraphics[scale=0.23]{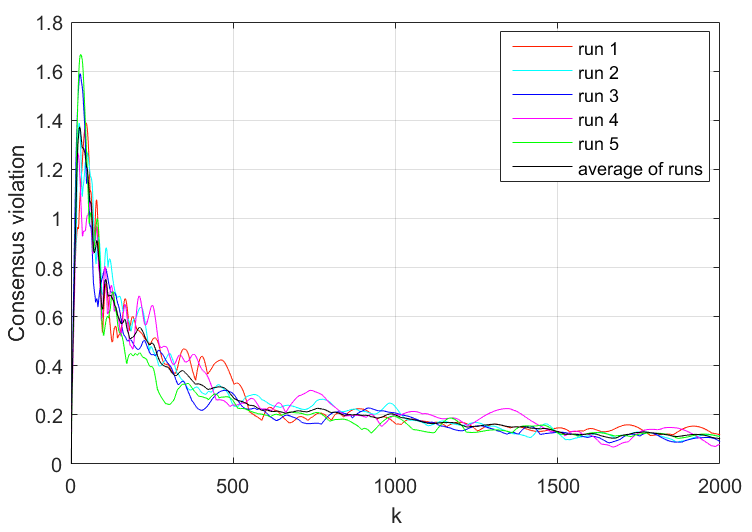}
\end{subfigure}
\begin{subfigure}[b]{0.33\textwidth}
\includegraphics[scale=0.23]{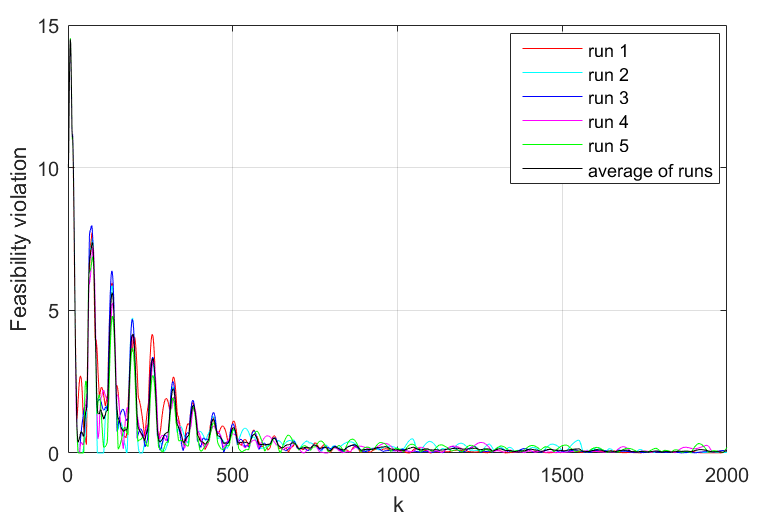}
\end{subfigure}
\begin{subfigure}[b]{0.33\textwidth}
\includegraphics[scale=0.23]{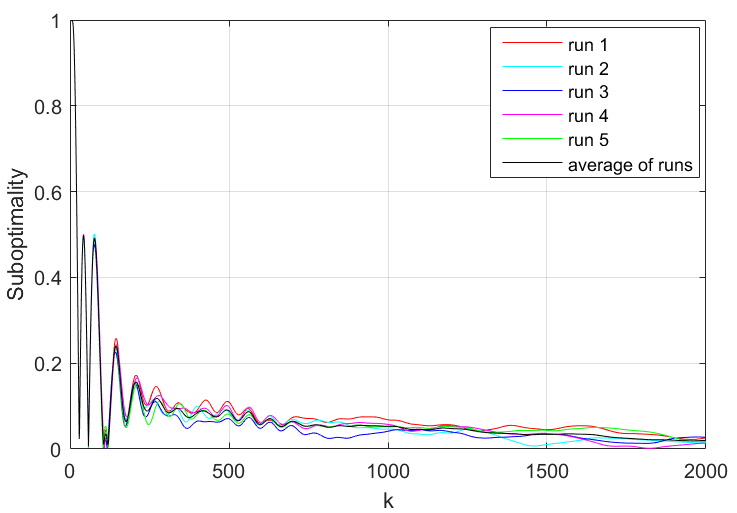}
\end{subfigure}
\vspace*{-3mm}
\caption{Static network topology with algebraic connectivity 0.05 where the first row corresponds to $C=2$ and the second row corresponds to $C=10$.}
\label{Static-Alg0.05}
\end{figure}

\begin{figure}[h]
\begin{subfigure}[b]{0.33\textwidth}
\includegraphics[scale=0.24]{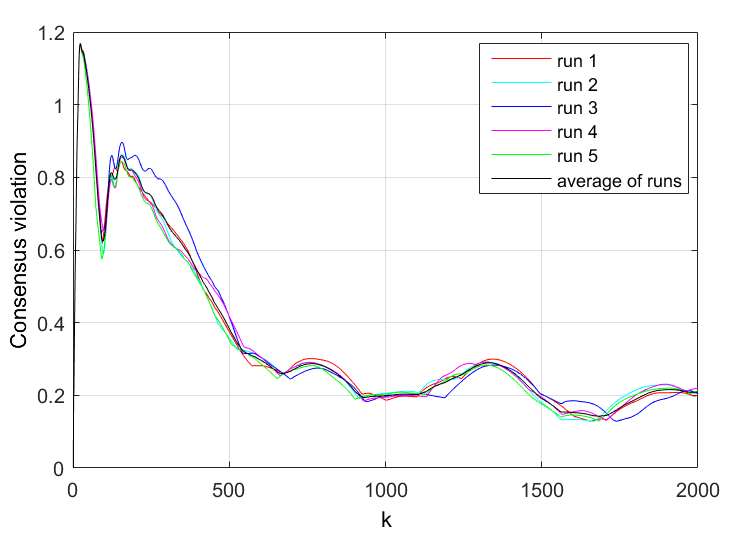}
\end{subfigure}
\begin{subfigure}[b]{0.33\textwidth}
\includegraphics[scale=0.24]{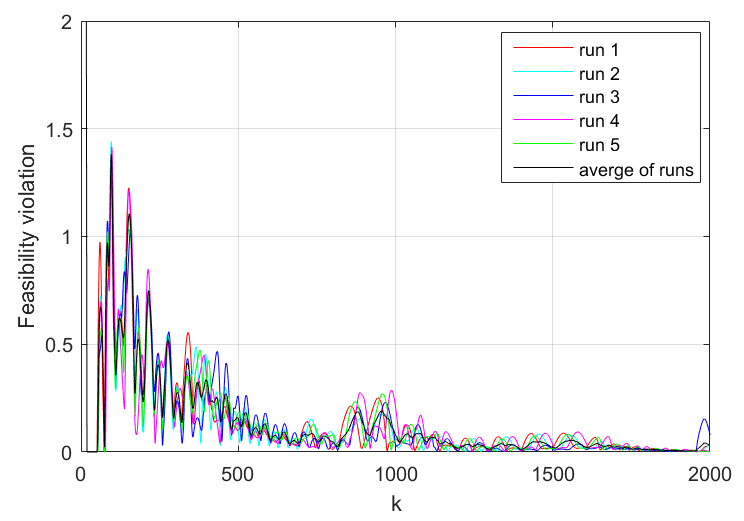}
\end{subfigure}
\begin{subfigure}[b]{0.33\textwidth}
\includegraphics[scale=0.245]{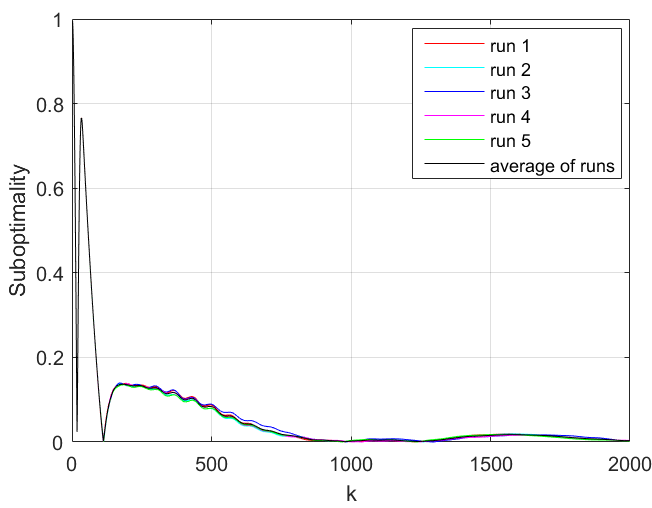}
\end{subfigure}
\\
\begin{subfigure}[b]{0.33\textwidth}
\includegraphics[scale=0.24]{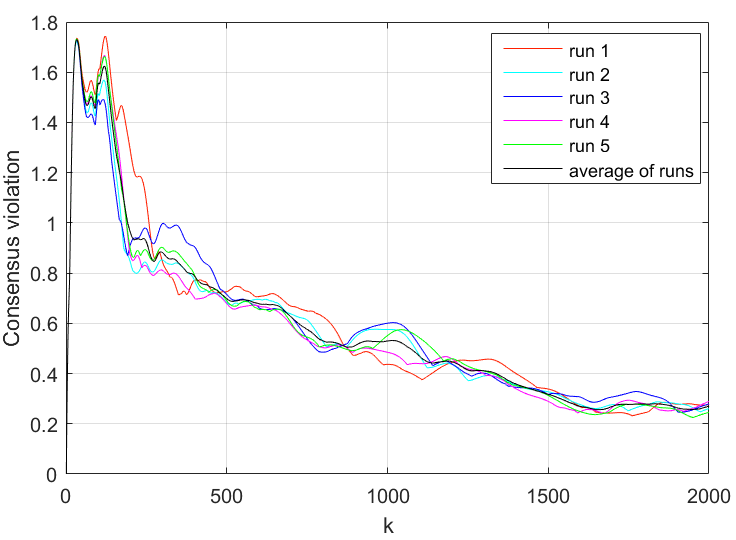}
\end{subfigure}
\begin{subfigure}[b]{0.33\textwidth}
\includegraphics[scale=0.24]{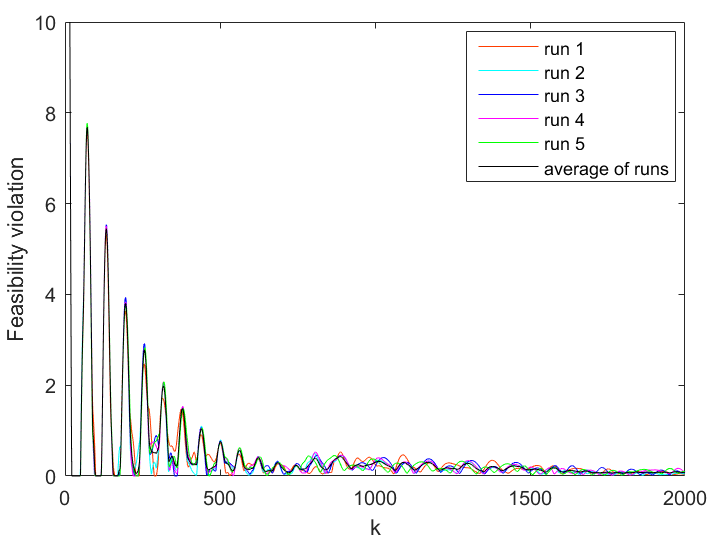}
\end{subfigure}
\begin{subfigure}[b]{0.33\textwidth}
\includegraphics[scale=0.24]{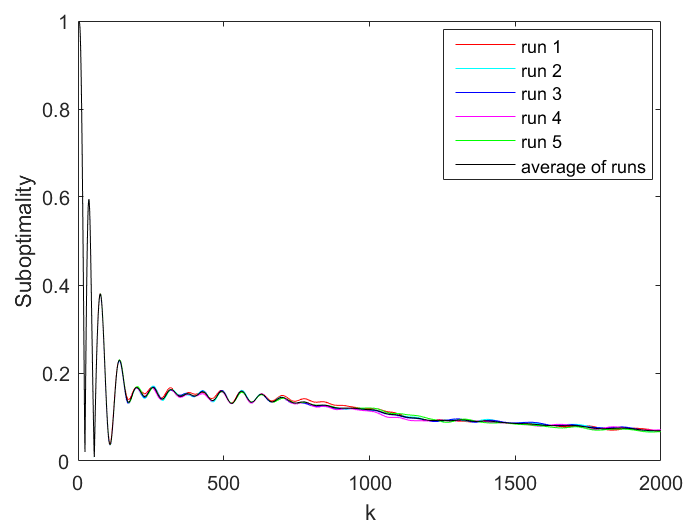}
\end{subfigure}
\caption{Dynamic network topology with algebraic connectivity 0.05 where the first row corresponds to $C=2$ and the second row corresponds to $C=10$.}
\label{Dynamic-Alg0.05}
\end{figure}

\begin{figure}
\begin{subfigure}[b]{0.33\textwidth}
\includegraphics[scale=0.23]{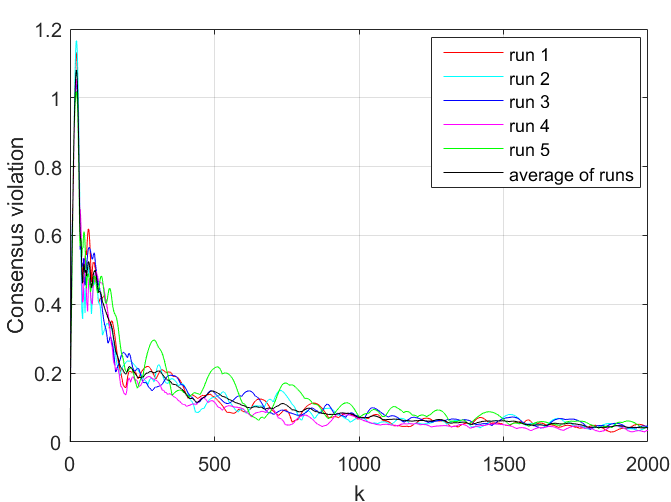}
\end{subfigure}
\begin{subfigure}[b]{0.33\textwidth}
\includegraphics[scale=0.35]{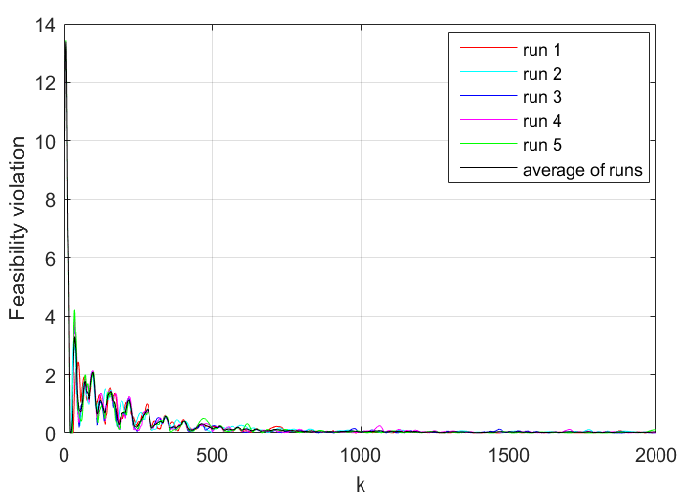}
\end{subfigure}
\begin{subfigure}[b]{0.33\textwidth}
\includegraphics[scale=0.235]{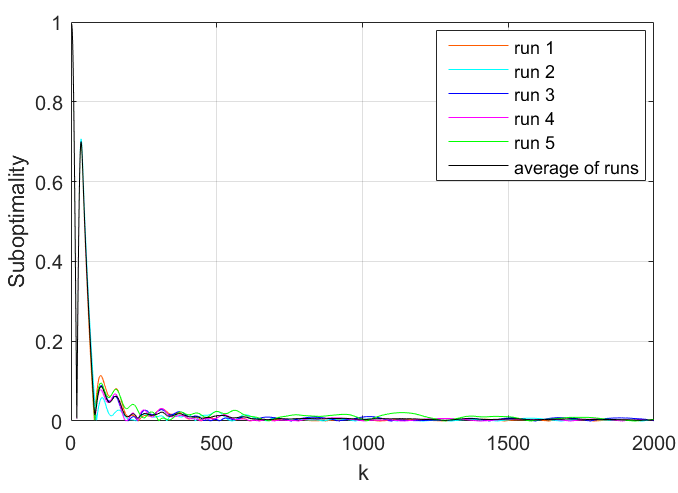}
\end{subfigure}
\\
\begin{subfigure}[b]{0.33\textwidth}
\includegraphics[scale=0.26]{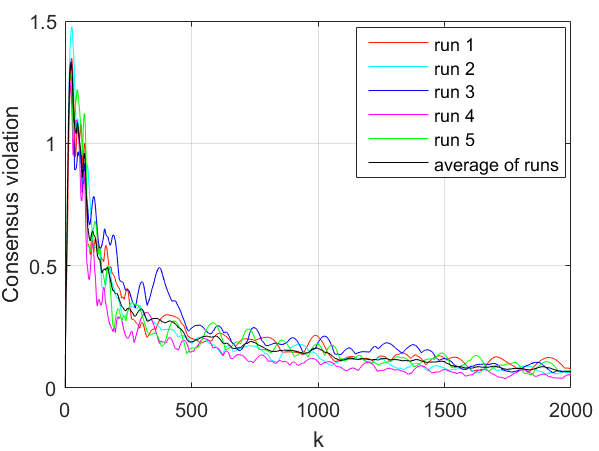}
\end{subfigure}
\begin{subfigure}[b]{0.33\textwidth}
\includegraphics[scale=0.27]{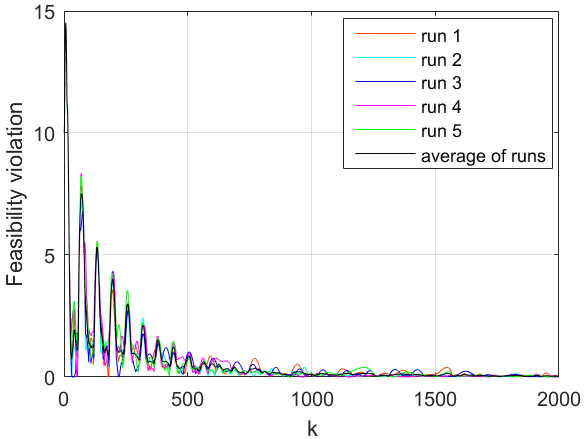}
\end{subfigure}
\begin{subfigure}[b]{0.33\textwidth}
\includegraphics[scale=0.265]{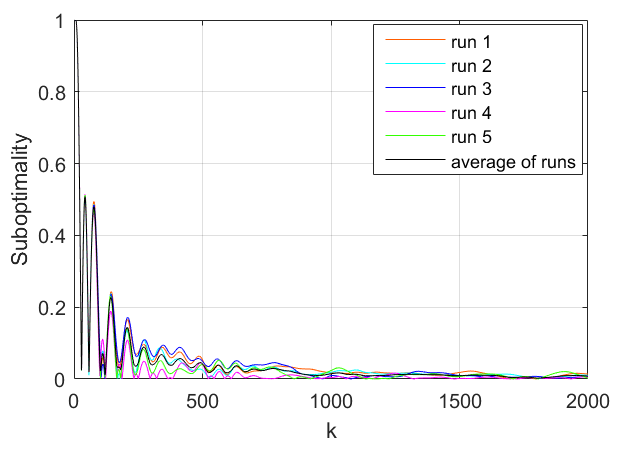}
\end{subfigure}
\caption{Static network topology with algebraic connectivity 1 where the first row corresponds to $C=2$ and the second row corresponds to $C=10$.}
\label{Static-Alg1}
\end{figure}

\begin{figure}
\begin{subfigure}[b]{0.33\textwidth}
\includegraphics[scale=0.25]{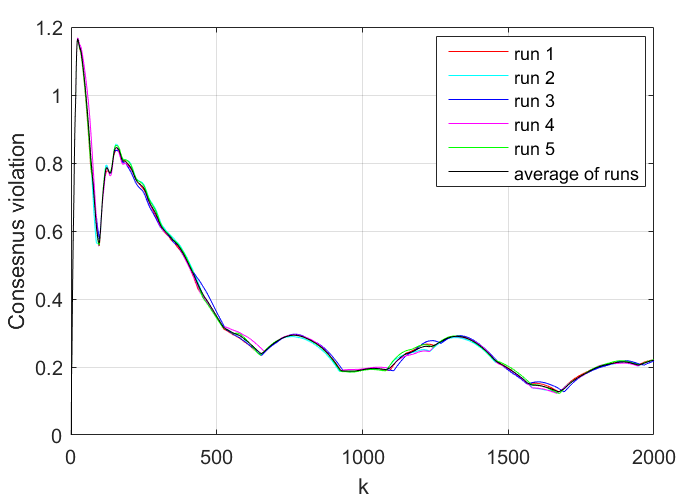}
\end{subfigure}
\begin{subfigure}[b]{0.33\textwidth}
\includegraphics[scale=0.245]{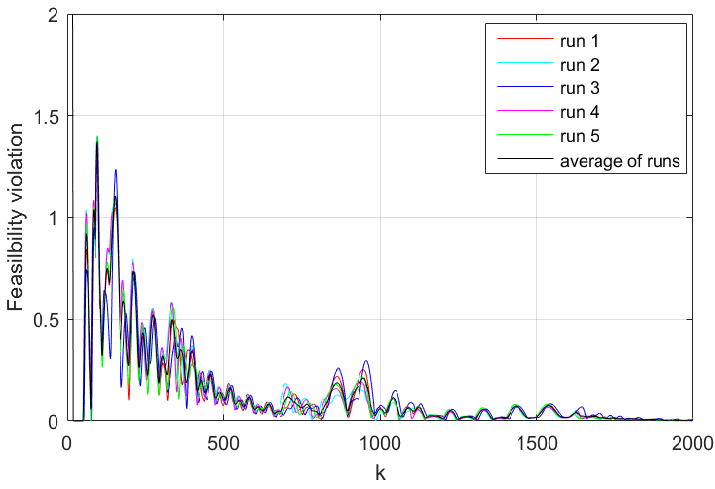}
\end{subfigure}
\begin{subfigure}[b]{0.33\textwidth}
\includegraphics[scale=0.25]{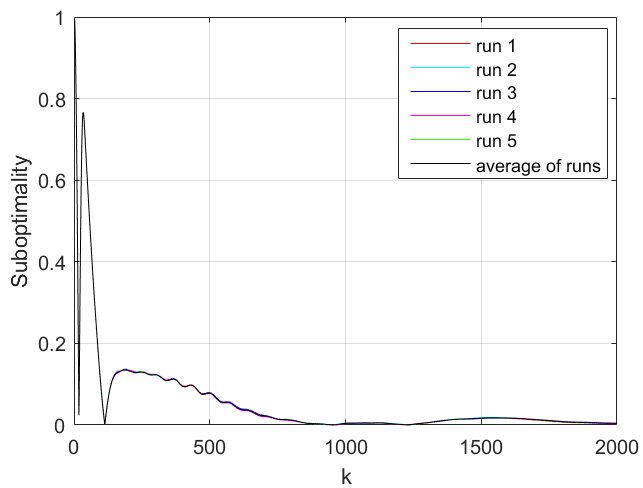}
\end{subfigure}
\\
\begin{subfigure}[b]{0.33\textwidth}
\includegraphics[scale=0.24]{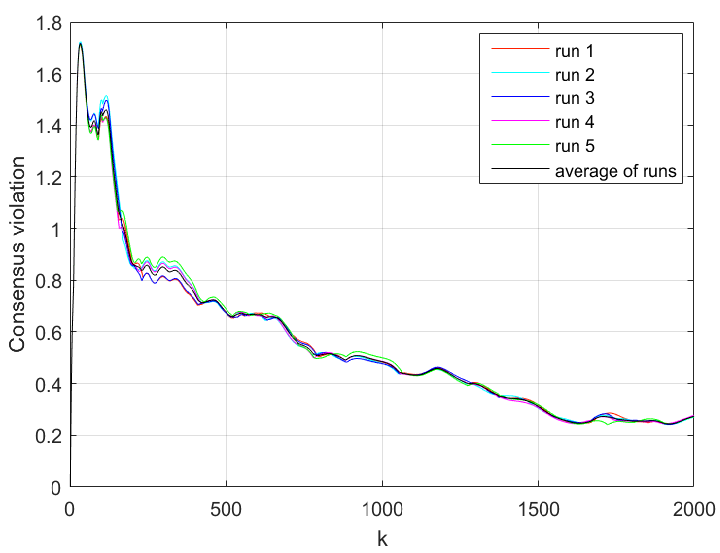}
\end{subfigure}
\begin{subfigure}[b]{0.33\textwidth}
\includegraphics[scale=0.24]{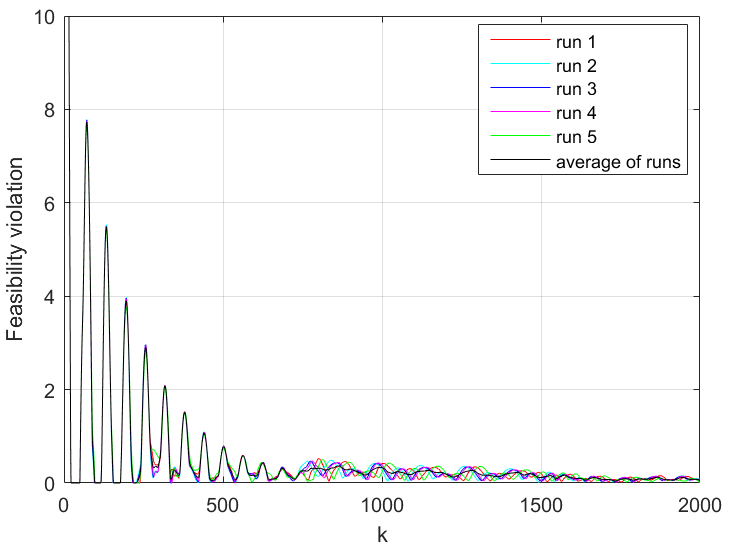}
\end{subfigure}
\begin{subfigure}[b]{0.33\textwidth}
\includegraphics[scale=0.24]{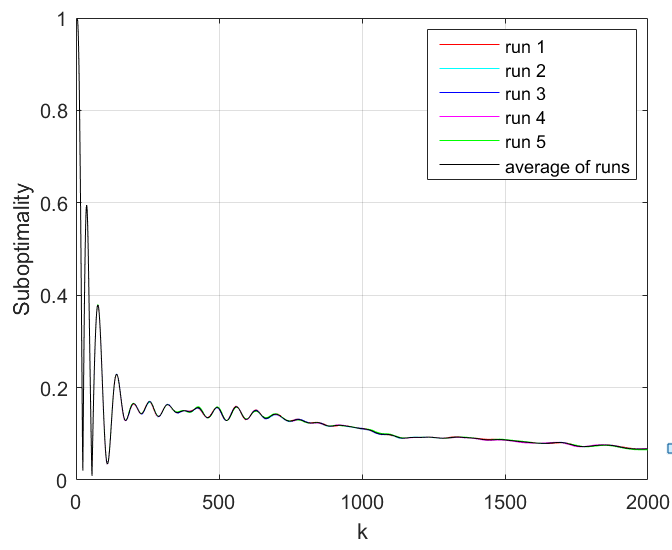}
\end{subfigure}
\caption{Dynamic network topology with algebraic connectivity 1 where the first row corresponds to $C=2$ and the second row corresponds to $C=10$.}
\label{Dynamic-Alg1}
\end{figure}
\begin{figure}[h]
\centering
\includegraphics[scale=0.3]{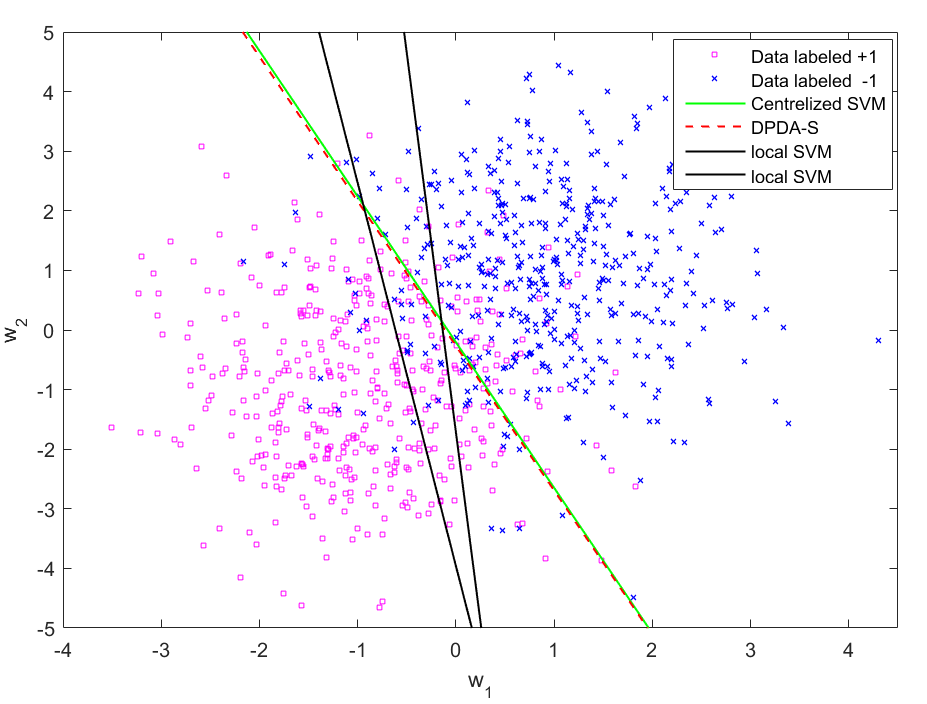}
\caption{Comparison among DPDA-S, local SVM, and central SVM for $C=10$ and algebraic connectivity $0.05$.}
\label{SVM}
\end{figure}
\clearpage
{\small
\singlespacing
\bibliographystyle{unsrt}
\bibliography{papers}
}
\newpage
\section{Appendix}
\subsection{Proof of Theorem~\ref{thm:general-rate}}
\begin{lemma}
\label{lem:pda-lemma}
Let $\mathbf{\bar{Q}}\triangleq\mathbf{\bar{Q}}(A,A_0)$, and $\bz=[\bx^\top \by^\top]^\top$ for $\bx\in\cX$, $\by\in\cY$. For any ${\bf x}\in \cX$, and $\by\in\cY$, the iterate sequence $\{\bz^k\}_{k\geq 1}$ defined as in the statement of Theorem~\ref{thm:general-rate} satisfies for all $k\geq 0$
{\small
\begin{align*}
\mathcal{L}&({\bf x}^{k+1},\by)-\mathcal{L}({\bf x},\by^{k+1})\leq\left[D_x(\bx,\bx^k)+D_y(\by,\by^k)-\fprod{T(\bx-\bx^k),~\by-\by^k}\right]\\
&-\left[D_x(\bx,\bx^{k+1})+D_y(\by,\by^{k+1})-\fprod{T(\bx-\bx^{k+1}),~\by-\by^{k+1}}\right]-\frac{1}{2}(\bz^{k+1}-\bz^k)^\top\mathbf{\bar{Q}}(\bz^{k+1}-\bz^k).
\end{align*}}%
\end{lemma}
\begin{proof}
Note that x-subproblem in \eqref{eq:PDA-x} is separable in local decisions $\{x_i\}_{i\in\cN}$; for each $i\in\cN$ the local subproblem over $x_i$ is strongly convex with constant $1/\tau_i$. Indeed, let $\bp^k=T^\top \by^k$ and define $\{p_i^k\}_{i\in\cN}$ such that $p_i^k$ is the subvector corresponding to the components of $x_i$, i.e., $\bp^k=[p_i^k]_{i\in\cN}$. Thus, the definitions of $\rho$, $f$ and $D_x$, $\nu_x=1$, and \eqref{eq:PDA-x} imply that for all $i\in\cN$
{\small
\begin{equation}
\label{eq:x-subproblem}
x_i^{k+1}=\argmin_{x_i}\rho_i(x_i)+f_i(x_i^k)+\fprod{\grad f_i(x_i^k),~x_i-x_i^k}+\fprod{p_i^k,x_i}+\frac{1}{2\tau_i}\norm{x_i-x_i^k}^2.
\end{equation}}%
Therefore, the strong convexity of the objective in local subproblem \eqref{eq:x-subproblem} for $i\in\cN$ implies
{\small
\begin{eqnarray*}
\lefteqn{\rho_i(x_i)+\fprod{\grad f_i(x_i^k),~x_i}+\fprod{p_i^k,x_i}+\frac{1}{2\tau_i}\norm{x_i-x_i^k}^2\geq}\nonumber\\
& & \rho_i(x_i^{k+1})+\fprod{\grad f_i(x_i^k),~x_i^{k+1}}+\fprod{p_i^k,x_i^{k+1}}+\frac{1}{2\tau_i}\norm{x_i^{k+1}-x_i^k}^2+\frac{1}{2\tau_i}\norm{x_i-x_i^{k+1}}^2.
\label{eq:triangular-ineq-Bregman}
\end{eqnarray*}}%
Convexity of $f_i$ and Lipschitz continuity of $\grad f_i$ implies that
{\small
\begin{equation*}
f_i(x_i)\geq f_i(x_i^k)+\fprod{\grad f_i(x_i^k),~x_i-x_i^k}\geq f_i(x_i^{k+1})+\fprod{\grad f_i(x_i^{k+1}),~x_i-x_i^{k+1}}-\frac{L_i}{2}\norm{x_i^{k+1}-x_i^k}^2.
\end{equation*}}%
Since $\sum_{i\in\cN}\fprod{p_i^k,~x_i}=\fprod{T\bx,~\by^k}$ for all $\bx$, summing these two inequalities 
for each $i\in\cN$, and then summing the resulting inequalities over $i\in\cN$, we get
{\small
\begin{eqnarray}
\lefteqn{\Phi(\bx)+D_x(\bx,\bx^{k})\geq}\label{eq:lemma-x}\\
& & \Phi(\bx^{k+1})+\fprod{T(\bx^{k+1}-\bx),~\by^k}+D_x(\bx,\bx^{k+1})+\tfrac{1}{2}(\bx^{k+1}-\bx^k)^\top\mathbf{\bar{D}}_\tau(\bx^{k+1}-\bx^k).
\nonumber
\end{eqnarray}}%
Similarly, let $\bq^k=T(2\bx^{k+1}-\bx^k)$ and define $q_0^k\in\reals^{m_0}$ and $q_i^k\in\reals^{m_i}$ for $i\in\cN$ such that $q_0^k$ is the subvector corresponding to the components of $\blambda$, and $q_i^k$ is the subvector corresponding to the components of $\theta_i$ for $i\in\cN$, i.e., $\bp^k=[{p_1^k}^\top \ldots {p_N^k}^\top {p_0^k}^\top]^\top$. Thus, the definitions of $h$ and $D_y$, and $\nu_y=1$ imply that according to \eqref{eq:PDA-x}
{\small
\begin{align*}
\blambda^{k+1}&=\argmin_{\blambda} h_0(\blambda)-\fprod{q_0^k,\blambda}+\frac{1}{2\gamma}\norm{\blambda-\blambda^k}^2,\\
\theta_i^{k+1}&=\argmin_{\theta_i}h_i(\theta_i)-\fprod{q_i^k,\theta_i}+\frac{1}{2\kappa_i}\norm{\theta_i-\theta_i^k}^2,\quad \forall\ i\in\cN.
\end{align*}}%
Therefore, the strong convexity of the objectives in these subproblems implies that
{\small
\begin{align*}
h_0(\blambda)-\fprod{q_0^k,\blambda}+\frac{1}{2\gamma}\norm{\blambda-\blambda^k}^2
&\geq h_0(\blambda^{k+1})-\fprod{q_0^k,\blambda^{k+1}}+\frac{1}{2\gamma}\norm{\blambda^{k+1}-\blambda^k}^2+\frac{1}{2\gamma}\norm{\blambda-\blambda^{k+1}}^2,\\
h_i(\theta_i)-\fprod{q_i^k,\theta_i}+\frac{1}{2\kappa_i}\norm{\theta_i-\theta_i^k}^2
&\geq h_i(\theta_i^{k+1})-\fprod{q_i^k,\theta_i^{k+1}}+\frac{1}{2\kappa_i}\norm{\theta_i^{k+1}-\theta_i^k}^2+\frac{1}{2\kappa_i}\norm{\theta_i-\theta_i^{k+1}}^2.
\end{align*}}%
Since $\fprod{q_0^k,~\blambda}+\sum_{i\in\cN}\fprod{q_i^k,~\theta_i}=\fprod{T(2\bx^{k+1}-\bx^k),~\by}$ for all $\by$, summing these the second inequality over $i\in\cN$ and then summing the resulting inequality with the first one, we get
{\small
\begin{eqnarray}
\lefteqn{h(\by)+D_y(\by,\by^{k})\geq}\label{eq:lemma-y}\\
& & h(\by^{k+1})-\fprod{T(2\bx^{k+1}-\bx^k),~\by^{k+1}-\by}+D_y(\by,\by^{k+1})+\tfrac{1}{2}(\by^{k+1}-\by^k)^\top\mathbf{D}_\kappa(\by^{k+1}-\by^k).
\nonumber
\end{eqnarray}}%
Summing \eqref{eq:lemma-x} and \eqref{eq:lemma-y} gives the desired result.
\end{proof}

Now we continue to the proof of Theorem~\ref{thm:general-rate}. Let $\mathbf{\bar{Q}}\triangleq\mathbf{\bar{Q}}(A,A_0)$. Since $\mathbf{\bar{Q}}\succeq 0$, we can drop the last term in the inequality given in the statement of Lemma~\ref{lem:pda-lemma}; and summing it over $k$, we get
{\small
\begin{align*}
\sum_{k=0}^{K-1}\mathcal{L}({\bf x}^{k+1},\by)-\mathcal{L}({\bf x},\by^{k+1})
\leq&\left[D_x(\bx,\bx^0)+D_y(\by,\by^0)-\fprod{T(\bx-\bx^0),~\by-\by^0}\right]\\
&-\left[D_x(\bx,\bx^{K})+D_y(\by,\by^{K})-\fprod{T(\bx-\bx^{K}),~\by-\by^{K}}\right].
\end{align*}}%
$\mathbf{\bar{Q}}\succeq 0$ also implies that $D_x(\bx,\bx^{K})+D_y(\by,\by^{K})-\fprod{T(\bx-\bx^{K}),~\by-\by^{K}}\geq 0$; therefore, \eqref{eq:DPA-rate} follows from Jensen's inequality.
\end{document}